\documentclass{elsarticle}

\usepackage{amssymb,latexsym,amsmath,amsthm,amscd}
\usepackage[all]{xy}
\xyoption{arc}
\providecommand{\XLeftMargin}{2cm}
\providecommand{\XTopMargin}{1.8cm}
\providecommand{\XRightMargin}{2cm}
\providecommand{\XBottomMargin}{1.8cm}

\usepackage[left=\XLeftMargin,top=\XTopMargin,right=\XRightMargin,bottom=\XBottomMargin]{geometry}
\usepackage{graphicx}
\usepackage{todonotes}

\newlength\XXXMyLength\makeatletter
\@ifclassloaded{amsart}{\def\Xadjustleft[#1]{}
}{
\def\Xadjustleft[#1]{\setlength\XXXMyLength{#1}\ifnum\numexpr\leftmargin>\numexpr\XXXMyLength\hspace{-\XXXMyLength}\else\hspace{-\leftmargin}\fi}
}
\makeatother

\UseComputerModernTips

\theoremstyle{plain}
\newtheorem{thm}{Theorem}[section]
\newtheorem*{thm*}{Theorem}
\newtheorem{lemma}[thm]{Lemma}

\newtheorem*{lem*}{Lemma}
\newtheorem{prop}[thm]{Proposition}
\newtheorem*{prop*}{Proposition}
\newtheorem*{claim}{Claim}
\newtheorem{cor}[thm]{Corollary}
\newtheorem*{cor*}{Corollary}

\newtheorem*{conj*}{Conjecture}

\theoremstyle{definition}

\newtheorem*{cons*}{Construction}
\newtheorem{df}[thm]{Definition}
\newtheorem*{df*}{Definition}
\newtheorem{nota}[thm]{Notation}
\newtheorem*{nota*}{Notation}
\newtheorem{qu}{Question}
\newtheorem*{qu*}{Question}

\newtheorem{rmk}[thm]{Remark}
\newtheorem*{rmk*}{Remark}

\newtheorem{ex}[thm]{Example}
\newtheorem*{ex*}{Example}

\newcommand{\bA}{\mathbb{A}}

\newcommand{\bC}{\mathbb{C}}

\newcommand{\bG}{\mathbb{G}}
\newcommand{\bH}{\mathbb{H}}

\newcommand{\bO}{\mathbb{O}}
\newcommand{\bP}{\mathbb{P}}
\newcommand{\bQ}{\mathbb{Q}}
\newcommand{\bR}{\mathbb{R}}
\newcommand{\bS}{\mathbb{S}}

\newcommand{\bZ}{\mathbb{Z}}

\newcommand{\cB}{\mathcal{B}}

\newcommand{\cD}{\mathcal{D}}
\newcommand{\cE}{\mathcal{E}}
\newcommand{\cF}{\mathcal{F}}

\newcommand{\cK}{\mathcal{K}}
\newcommand{\cL}{\mathcal{L}}

\newcommand{\cO}{\mathcal{O}}

\newcommand{\cQ}{\mathcal{Q}}

\newcommand{\cT}{\mathcal{T}}
\newcommand{\cU}{\mathcal{U}}
\newcommand{\cV}{\mathcal{V}}
\newcommand{\cW}{\mathcal{W}}

\newcommand{\cY}{\mathcal{Y}}

\newcommand{\fg}{\mathfrak g}

\newcommand{\fk}{\mathfrak k}

\newcommand{\fo}{\mathfrak o}
\newcommand{\fp}{\mathfrak p}

\newcommand{\fs}{\mathfrak s}

\newcommand{\fu}{\mathfrak u}

\DeclareMathOperator{\Ad}{Ad}

\DeclareMathOperator{\Aut}{Aut}
\DeclareMathOperator{\Br}{Br}

\DeclareMathOperator{\ch}{ch}
\DeclareMathOperator{\CH}{CH}

\DeclareMathOperator{\End}{End}

\DeclareMathOperator{\Gal}{Gal}

\DeclareMathOperator{\Grass}{Gr}
\DeclareMathOperator{\Hom}{Hom}

\DeclareMathOperator{\Id}{Id}
\DeclareMathOperator{\im}{im}

\DeclareMathOperator{\Isom}{Isom}

\DeclareMathOperator{\Ker}{Ker}

\DeclareMathOperator{\Lie}{Lie}
\DeclareMathOperator{\Mat}{M}

\DeclareMathOperator{\Proj}{Proj}

\DeclareMathOperator{\rank}{rank}

\DeclareMathOperator{\Res}{Res}
\DeclareMathOperator{\spann}{span}
\DeclareMathOperator{\Spec}{Spec}

\DeclareMathOperator{\Stab}{Stab}

\DeclareMathOperator{\td}{td}

\DeclareMathOperator{\Tr}{Tr}
\DeclareMathOperator{\tr}{tr}
\DeclareMathOperator{\uhp}{\bH}
\DeclareMathOperator{\Vol}{Vol}

\DeclareMathOperator{\so}{\fs\fo}
\DeclareMathOperator{\Gl}{GL}

\DeclareMathOperator{\Sl}{SL}
\DeclareMathOperator{\SO}{SO}

\DeclareMathOperator{\SU}{SU}

\DeclareMathOperator{\PGl}{PGL}
\DeclareMathOperator{\Orth}{O}
\DeclareMathOperator{\Spin}{Spin}
\DeclareMathOperator{\GSpin}{GSpin}
\DeclareMathOperator{\GO}{GO}

\DeclareMathOperator{\Clif}{C}

\newcommand{\surjects}{\twoheadrightarrow}
\newcommand{\injects}{\hookrightarrow}

\newcommand{\mtx}[4]{\left(\begin{matrix}#1&#2\\#3&#4\end{matrix}\right)}
\newcommand{\smtx}[4]{\left(\begin{smallmatrix}#1&#2\\#3&#4\end{smallmatrix}\right)}

\def\emphh{\textbf}

\newcommand{\inddef}[1]{\emphh{#1}}

\newcommand{\inddefs}[2]{\emphh{#1}}
\newcommand{\inddefsalpha}[3]{\emphh{#1}}
\newcommand{\inddefss}[2]{\emphh{#1}}
\newcommand{\inddefssalpha}[3]{\emphh{#1}}
\newcommand{\inddefns}[2]{\emphh{#1}}
\newcommand{\inddefnss}[2]{\emphh{#1}}
\newcommand{\indnota}[1]{\ensuremath{#1}}
\newcommand{\indnotalpha}[2]{\ensuremath{#1}}

\def\sumprime{\mathop{\sum{\raise3pt\hbox{${}'$}}}}

\makeatletter
\def\revddots{\mathinner{\mkern1mu\raise\p@
\vbox{\kern7\p@\hbox{.}}\mkern2mu
\raise4\p@\hbox{.}\mkern2mu\raise7\p@\hbox{.}\mkern1mu}}
\makeatother

\newcommand{\floor}[1]{\left\lfloor #1 \right\rfloor}

\providecommand{\abs}[1]{\left\vert #1 \right\vert}

\newcommand{\comment}[1]{}

\newcommand{\GL}{\Gl}
\newcommand{\SL}{\Sl}

\begin{document}

\title{Toroidal Compactifications and Dimension Formulas for Spaces of Modular Forms for Orthogonal Shimura Varieties}
\author{Andrew Fiori}

\ead{afiori@mast.queensu.ca}
\address{Department of Mathematics and Statistics, Queens University, Jeffery Hall, University Ave, Office 517, Kingston, ON, K7L 3N6, Canada}

\begin{abstract}
In this paper we describe the general theory of constructing toroidal compactifications of locally symmetric spaces and using these to compute dimension formulas for spaces of modular forms. We focus explicitly on the case of the orthogonal locally symmetric spaces arising from quadratic forms of signature $(2,n)$, giving explicit details of the constructions, structures and results in these cases. 

This article does not give explicit cone decompositions, compute explicit intersection pairings, or count cusps and thus does not give any complete formulas for the dimensions.
\end{abstract}


\maketitle

This article contains material originally appearing in my Ph.D. thesis \cite{Fiori_PHD} as well as some new work done on the problem since then.

\section{Introduction}

The primary object of interest in this document are Shimura varieties of orthogonal type. In order to give a satisfactory definition of these one needs the terminology and notation of the theory of Hermitian symmetric spaces \cite{Helgason}, quadratic spaces and orthogonal groups \cite{Omeara, Lam_QF, Serre_arith}. 
To put it in the right context one should perhaps also have access to the basic notions of Shimura varieties \cite{Milne_shimura, Deligne_Shimura1}. 

It is our intent in this 
document
to give a survey of the basic theory of orthogonal symmetric spaces, their toroidal compactifications, and an approach for computing dimensions of spaces of modular forms. Other references include \cite{ Fiori_MSC,Brunier_123}.
The sections of this 
document
are organized as follows.
\begin{itemize}
\item[(\ref{sec:SectionBasics})] Introduces key notations and results for orthogonal groups.

\item[(\ref{sec:SectionHermitian})] Covers the key notions of Hermitian symmetric domains.

\item[(\ref{sec:SectionModular})] Provides a basic definition of modular forms.

\item[(\ref{sec:SectionToroidal})] Surveys the construction of toroidal compactifications explaining the relevant structures for the orthogonal group. We do not however give any explicit compactifications for this case.

\item[(\ref{sec:SectionDimension})] Surveys the problem of computing dimension formulas for spaces of modular forms via the Hirzebruch-Mumford proportionality theorem (see \cite{Mumford_Proportionality}).

\item[(\ref{sec:SectionRamification})] Discusses the ramification structures between different levels introducing two interesting classes of cycles on orthogonal Shimura varieties.
\end{itemize}
One can view Sections (\ref{sec:SectionBasics}) - (\ref{sec:SectionToroidal}) as an overview of the construction of toroidal compactifications (see \cite{AMRT}).
This is done somewhat in the spirit of \cite{Namikawa}, who covers the theory for the Siegel spaces, except we focus on the orthogonal case.
Sections (\ref{sec:SectionDimension})-(\ref{sec:SectionRamification}) form a survey on an approach to the problem of finding dimension formulas for spaces of modular forms. The approach follows that of \cite{Tsushima}

{\def\XMetaCompile{1}

\section{Basics of Orthogonal Groups}
\label{sec:SectionBasics}

It is natural to assume that the reader has a basic understanding of quadratic spaces. Thus, the main purpose of this section is to introduce our notation.
\begin{df}
Let $R$ be an integral domain, and $K$ be its field of fractions.
Given a finitely generated $R$-module $V$, a \inddef{quadratic form} on $V$ is a mapping $q:V\rightarrow K$ such that:
\begin{enumerate}
\item $q(r\vec{x})=r^2q(\vec{x})$ for all $r\in K$ and $\vec{x}\in V$, and
\item $B(\vec{x},\vec{y}) := q(\vec{x}+\vec{y})-q(\vec{x})-q(\vec{y})$ is a bilinear form.
\end{enumerate}
Given such a pair $\indnotalpha{(V,q)}{Vq}$, we call $V$ a \inddef{quadratic module} over $R$.
The quadratic module $V$ is said to be \inddefs{regular}{quadratic module} or \inddefs{non-degenerate}{quadratic module} if for all $\vec{x}\in V$ there exists $\vec{y}\in V$ such that $B(\vec{x},\vec{y})\neq 0$.
\end{df}
\begin{rmk}
Given an $R$ module $V$ and a bilinear form $b:V\times V \rightarrow K$ we have an associated quadratic form $\indnota{q(\vec{x})} = b(\vec{x},\vec{x})$. Note that $\indnota{B(\vec{x},\vec{y})} = 2\indnota{b(\vec{x},\vec{y})}$.
\end{rmk}

\begin{df}
We define the \inddefnss{Clifford algebra}{quadratic form} and the \emphh{even Clifford algebra} to be respectively:
\[ \indnotalpha{\Clif_q}{Clifq} := \underset{k}\oplus V^{\otimes k}/(\vec{v}\otimes \vec{v} - q(\vec{v}))\text{ and } \indnotalpha{\Clif_q^0}{Clifq0} := \underset{k}\oplus V^{\otimes 2k}/(\vec{v}\otimes \vec{v} - q(\vec{v})). \]
They are isomorphic to matrix algebras over quaternion algebras. We denote the standard involution $\vec{v}_1\otimes\cdots\otimes \vec{v}_m \mapsto \vec{v}_m\otimes\cdots \otimes \vec{v}_1$ by $v\mapsto v^\ast$.
To a quadratic form $q$ we will associate the following algebraic groups:
\begin{align*}
  \indnotalpha{\Orth_q}{Oq}(R') &= \{ g\in \GL(V\otimes_R R') \mid q(\vec{x})=q(g(\vec{x}))\text{ for all } x\in V\otimes_R R' \}\\
  \indnotalpha{\SO_q}{SOq}(R') &= \{ g \in \Orth_q(R') \mid  \det(g) = 1 \} \\
  \indnotalpha{\GSpin_q}{GSpinq}(R') &=   \{ g \in (\Clif^0_q\otimes_R R')^\times\mid  gVg^{-1} \subset V \} \\
  \indnotalpha{\Spin_q}{Spinq}(R') &= \{ g \in \GSpin_q(R') \mid g\cdot g^\ast = 1  \}.
\end{align*}
\end{df}

\begin{prop}
Given a quadratic form $q$ we have a short exact sequence of algebraic groups:
\[ 0 \rightarrow \underline{\bZ/2\bZ} \rightarrow \Spin_q \rightarrow \SO_q \rightarrow 0.  \] 
Over a number field $k$, with $\Gamma = \Gal(\overline{k}/k)$, this becomes the long exact sequence:
\[ 0 \rightarrow \bZ/2\bZ \rightarrow \Spin_q(k) \rightarrow \SO_q(k) \overset{\theta} \rightarrow H^1(\Gamma, \bZ/2\bZ) \rightarrow \ldots. \]
The map $\theta$ is called the \inddef{spinor norm}.
\end{prop}

\begin{nota}
We have the following standard invariants of $(V,q)$:
\begin{itemize}
\item Whenever $V$ is free over $R$ we shall denote by $\indnota{D(q)}$ the \inddefss{discriminant}{quadratic form} of $q$, that is, $D(q)=\det( b(\vec{v}_i,\vec{v}_j)_{i,j} )$ for some choice of basis $\{\vec{v}_1,\ldots,\vec{v}_n \}$.
\item 
We shall denote by $\indnota{H(q)}$ (or $H_R(q),H_\fp(q)$) the \inddefss{Hasse invariant}{quadratic form} of $q$,  that is, if over the field of fractions $K$ of $R$ we may express $q(\vec{x}) = \sum_i a_ix_i^2$ then $H(q) = \prod_{i<j} (a_i,a_j)_K$. Here $\indnotalpha{(a,b)_K}{abk}$ denotes the Hilbert symbol (see \cite[Ch. III]{Serre_arith} and \cite[Ch. XIV]{Serre_local}).
\item 
We shall denote by $\indnota{W(q)}$ the \inddefss{Witt invariant}{quadratic form} of $q$, that is, the class in $\Br(K)$ of $\Clif_q$ when $\dim(V)$ is odd or of $\Clif^0_q$ when $\dim(V)$ is even.
\item 
For a real place, $\rho: R \rightarrow \bR$, we shall denote by $\indnotalpha{(r_\rho,s_\rho)_\rho}{rs}$ the \inddefss{signature}{quadratic form} of $q$ at $\rho$. Here $r_\rho$ denotes the dimension of the maximal positive-definite subspace of $V\otimes_\rho\bR$ and $s_\rho$ denotes the dimension of the maximal negative-definite subspace of $V\otimes_\rho \bR$.
\end{itemize}
\end{nota}

}

{\def\XMetaCompile{1}

\section{Hermitian Symmetric Spaces}
\label{sec:SectionHermitian}

In this section we briefly recall some key results about Hermitian symmetric spaces. A good reference on this topic is \cite{Helgason}.  Most of what we will use can also be found in \cite[Sec I.5]{BorelJi}, or \cite[Sec. 3.2]{AMRT}.

\begin{df}
A \inddef{symmetric space} is a Riemannian manifold $\indnotalpha{\cD}{Dc}$ such that for each $x\in \cD$ there exists an isometric involution $s_x$ of $\cD$ for which $x$ is locally the unique fixed point.
We say that $\cD$ is \inddefs{Hermitian}{symmetric space} if $\cD$ has a complex structure making $\cD$ Hermitian.
\end{df}

\begin{ex}
The standard example of this is the upper half plane:
\[ \indnotalpha{\uhp}{H} = \{ x+iy \in \bC \mid y > 0 \}. \]
\end{ex}

It is a consequence of the definition that we have:
\begin{thm}
Fix $x\in \cD$, $G=\Isom(\cD)^0$, $K=\Stab_G(x)$ and let $s_x$ act on $G$ by conjugation then $\cD\simeq G/K$ and $(G^{s_x})^0 \subset K \subset G^{s_x}$.
Moreover, given any real Lie group $G$, an inner automorphism $s:G\rightarrow G$ of order $2$, and $K$ such that $(G^s)^0 \subset K \subset G^s$, then the manifold $\cD=G/K$ is a symmetric space.
\end{thm} 
See \cite[Thm. IV.3.3]{Helgason}.

\begin{thm}
A symmetric space $\cD = G/K$ is Hermitian if and only if the centre $Z(K)$ of $K$ has positive dimension.
Moreover, if $\cD$ is irreducible then $Z(K)^0=\SO_2(\bR)$.
\end{thm}
See \cite[Thm. VIII.6.1]{Helgason}.

There are three main types of symmetric spaces:
\begin{enumerate}
\item Compact Type:
      In general these come from compact Lie groups $G$.
\item Non-Compact Type:
      In general these arise when $K^0$ is the maximal compact connected Lie subgroup of $G$, or equivalently when $s_x$ is a Cartan involution.
\item Euclidean Type:
      These generally arise as quotients of Euclidean space by discrete subgroups.
\end{enumerate}
The definitions of these types can be made precise by looking at the associated Lie algebras.

\begin{claim}
Every symmetric space decomposes into a product of the three types listed above.
\end{claim}
See \cite[Ch. V Thm. 1.1]{Helgason}.

For $\cD$ a Hermitian symmetric space of the non-compact type, one often considers the following objects (see \cite{Helgason} for details):
\begin{itemize}
\item The Lie algebra $\fg$ of $G$.
\item The Lie sub-algebra $\fk\subset\fg$ of $K$.
\item The Killing form $B(X,Y) = \Tr(\Ad(X)\circ\Ad(Y))$ on $\fg$.
\item The orthogonal complement $\fp$ of $\fk$ under $B$ is identified with the tangent space of $\cD$.
\item The centre $Z(K)$ of $K$ and its Lie algebra $\fu$.
\item A map $h_0: \SO_2 \rightarrow Z(K) \subset K \subset G$ such that $K$ is the centralizer of $h_0$.
\item The element $s=\Ad(h_0(e^{i\pi/2}))$ induces the Cartan involution whereas the element $J=\Ad(h_0(e^{i\pi/4}))$ induces the complex structure.
\end{itemize}

Through these one can construct:
\begin{itemize}
\item A $G$-invariant metric on $\cD$ (via $B$ and the identification of the tangent space of $\cD$ with $\fp$).
\item The dual Lie algebra $\fg^\ast = \fk \oplus i\fp$. This is the Lie algebra of $\indnotalpha{\breve{G}}{Gbreve}$ the compact real form of $G$.
\item The ideals $\fp_+, \fp_- \subset \fp_\bC$ which are the eigenspaces of $\fu$.
\item The parabolic subgroups $\indnota{P_\pm}$ associated respectively to $\indnotalpha{\fp_\pm}{pfpm}$.
\item The embeddings $\cD = G/K \injects G_\bC/K_\bC P_- \simeq \breve{G}/K \simeq \breve{\cD}$.
\end{itemize}

There exists a duality between the compact and non-compact types, that is, if $\cD$ is of the compact type, there exists a dual symmetric space $\breve{\cD}$ such that $\cD\injects \breve{\cD}$. 
The following theorem makes this more precise.

\begin{thm}
The subgroups $P_\pm$ and $K$ defined above satisfy the following:
\begin{itemize}
\item The natural map $P_+\times K_\bC\times P_- \rightarrow G_\bC$ is injective, and the image contains $G$.

\item There exists  holomorphic mappings
\[ \xymatrix{ 
\cD \ar@{^{(}->}[r]\ar@{=}[d]    &     \fp_+                          \ar@{->}[r]^{exp}    &  P_+        \ar@{=}[d]    \\
G/K \ar@{->}[r]     &     P_+\times K_\bC\times P_-/(K_\bC P_-) \ar@{=}[d]  \ar@{->>}[r] &        P_+  \ar@{->>}[d]\\
P_+ \ar@{->>}[r]&    P_+\times K_\bC\times P_-/(K_\bC P_-)    \ar@{->>}[r]&   G_\bC/(K_\bC P_-)
}
\]
\item These embed $\cD$ into the complex projective variety $\breve{\cD}\simeq G_\bC/(K_\bC P_-)$. Moreover, the inclusion $\cD$ in $\fp_+$ realizes the space as a bounded domain.
\end{itemize}
\end{thm}
See \cite[Thm. 1 Sec. 3.2]{AMRT} or \cite[Sec. VIII.7]{Helgason}.

We wish to describe the image of $\cD$ in $\fp_+$. To this end we have the following result.
\begin{thm}
With the notation as above, where $r$ is the $\bR$-rank of $G$, there exists a morphism
$\varphi : \SU_2 \times\SL_2^r \rightarrow G$ such that:
\begin{enumerate}
\item $\varphi(u,h_0^{SL}(u)^r) = h_0(u)$, and
\item $\varphi$ induces a map $\uhp^r\rightarrow \cD$.
\end{enumerate}
Moreover, every symmetric space map $\uhp \rightarrow \cD$ factors through $\varphi$.
\end{thm}
See \cite[Thm. 2 Sec. 3.2]{AMRT} and \cite[Prop 2  Sec. 3.2]{AMRT}.

Let $\tau$ denote complex conjugation with respect to $\fg^c$ then:
\[ B_\tau(u,v) = -B(u,\tau(v)) \] 
is a positive-definite Hermitian form on $\fg_\bC$.
For each $X\in\fp_+$ we have a map $[\cdot,X] : \fp_-\rightarrow \fk_\bC$. 
Denote by $[\cdot,X]^\ast : \fk_\bC\rightarrow \fp_-$ the adjoint with respect to this Hermitian pairing.
We may now state the following.
\begin{thm}
The image of $\cD\in \fp_+$ is:
\[ 
       \Ad(K)\cdot \im(\varphi) = 
\{X \mid [\cdot,X]^\ast\circ [\cdot,X] < 2\Id_{\fp_-} \}, \]
where the inequality implies a comparison of operator norms.
\end{thm}
See \cite[Thm. 3 Sec. 3.2]{AMRT}.

\begin{cor}
Every Hermitian symmetric domain $\cD$ of the non-compact type can be realized as a bounded domain.
\end{cor}
See \cite[Thm. VIII.7.1]{Helgason}.

\subsection{The O(2,n) Case}

We now discuss the example of the Hermitian symmetric spaces in which we are most interested. That is those associated to quadratic spaces of signature $(2,n)$.
Other references on this topic include \cite{Fiori_MSC,Brunier_123,Brunier_BP}.

Let $(V,q)$ be a quadratic space over $\bQ$. Then $V(\bR) := V\otimes \bR$ has signature $(r,s)$ for some choice of $r,s$.
The maximal compact subgroup of $\Orth_q(\bR)$ is $K \simeq \Orth_r(\bR)\times\Orth_s(\bR) \subset \Orth_q(\bR)$ and $\Orth_q(\bR)/K$ is a symmetric space. These only have complex structures (and thus are Hermitian) if one of $r$ or $s$ is $2$. Since interchanging $r$ and $s$ does not change the orthogonal group (it amounts to replacing $q$ by $-q$) we will assume that $r=2$.
We wish to construct the associated symmetric spaces along with its complex structure in this case.
\begin{rmk}
For much of the following discussion only the $\bR$-structure will matter, and as such, the only invariants of significance are the values $r$ and $s$. However, when we must consider locally symmetric spaces and their compactifications the $\bQ$-structure, and potentially the $\bZ$-structure, will become relevant.
\end{rmk}

\subsubsection{The Grassmannian}
Let $(V,q)$ be of signature $(2,n)$. We consider the Grassmannian of 2-dimensional subspaces of $V(\bR)$ on which the quadratic form $q$ restricts to a positive-definite form, namely:
\[ \indnotalpha{\Grass(V)}{Grass} := \{v \subset V \mid \dim(v) = 2, q|_v > 0\}. \]

By Witt's extension theorem (see \cite[Thm. IV.3]{Serre_arith}), the group $G=\Orth_q(\bR)$ will act transitively on $\Grass(V)$.
If we fix $v_0\in \Grass(V)$ then its stabilizer $K_{v_0}$ will be a maximal compact subgroup. Indeed, since this group must preserve both the plane and its orthogonal complement we have $K_{v_0} \simeq \Orth_2\times \Orth_n$. Thus $\Grass(V) = G/K_{v_0}$ realizes a symmetric space.

\begin{rmk}
Though this is a simple and useful realization of the space, it is not clear from this construction what the complex structure should be.
\end{rmk}

\subsubsection{The Projective Model}

We consider the complexification $V(\bC)$ of the space $V$ and the projectivization $P(V(\bC))$.
We then consider the zero quadric:
\[ N := \{ [\vec{v}]\in P(V(\bC)) \mid b(\vec{v},\vec{v}) = 0\} .\]
It is a closed algebraic subvariety of projective space. We now define:
\[\indnotalpha{\kappa}{kappa} := \{ [\vec{v}]\in P(V(\bC)) \mid b(\vec{v},\vec{v}) = 0, b(\vec{v},\overline{\vec{v}})>0\}. \]
This is a complex manifold of dimension $n$ consisting of $2$ connected components.
\begin{rmk}
One must check that these spaces are in fact well defined, that is, that the conditions do not depend on a representative $\vec{v}$.
Indeed $b(c\vec{v},c\vec{v}) = c^2b(\vec{v},\vec{v})$ and $b(c\vec{v},\overline{c\vec{v}}) = c\overline{c}b(\vec{v},\overline{\vec{v}})$.
\end{rmk}

\begin{rmk}
The orthogonal group $\Orth_q(\bR)$ acts transitively on $\kappa$. In order to see this we reformulate the condition that $\vec{v} = X+iY \in V(\bC)$ satisfies $[\vec{v}]\in\kappa$ as follows. We observe that:
\begin{align*}
 b(X+iY,X+iY) &=b(X,X)-b(Y,Y) + 2ib(X,Y) \text{ and } \\b(X+iY,X-iY) &= b(X,X)+b(Y,Y).
\end{align*}
 It follows from the conditions
 $b(X+iY,X+iY) = 0$ and $b(X+iY,X-iY) > 0$ that:
 \[ [\vec{v}]\in\kappa \Leftrightarrow b(X,X)=b(Y,Y)>0 \text{ and } b(X,Y)=0.\]
We thus have that $\Orth_q(\bR)$ acts on $\kappa$.
To show that it acts transitively we appeal to Witt's extension theorem to find $g \in \Orth_q(\bR)$ taking $X\mapsto X'$ and $Y\mapsto Y'$. This isometry $g$ then maps $[\vec{v}] \mapsto [\vec{v}']$.
\end{rmk}

Consider the subgroup $\Orth_q^+(\bR)$ of elements whose spinor norm equals the determinant. This consists of those elements which preserve the orientation of any, and hence all, positive-definite planes. The group $\Orth_q^+(\bR)$ preserves the $2$ components of $\kappa$ whereas $\Orth_q\setminus \Orth_q^+(\bR)$ interchanges them. 
Pick either component of $\kappa$ and denote it $\kappa^+$.

\begin{prop}
The assignment $[\vec{v}] \mapsto v(\vec{v}) := \bR X + \bR Y$ gives a real analytic isomorphism $\kappa^+ \rightarrow \Grass(V)$.
\end{prop}
This is a straightforward check (see \cite[Lem. 2.3.38]{Fiori_MSC}).

\subsubsection{The Tube Domain Model}
\label{sec:SectionTubeDomain}

Pick $e_1$ an isotropic vector in $V(\bR)$ and pick $e_2$ such that $b(e_1,e_2) = 1$. Define $\indnotalpha{\cU}{Uc} := V \cap e_2^\perp \cap e_1^\perp$. We then may express elements of $V(\bC)$ as $(a,b,\vec{y})$, where $a,b\in\bC$ and $\vec{y}\in\cU$. Thus
\[V = \bQ e_1 \oplus \bQ e_2 \oplus \cU   \]
and $\cU$ is a quadratic space of type $(1,n-1)$.

\begin{df}
We define the tube domain 
\[ \indnotalpha{\uhp_q}{Hq} := \{\vec{y}\in \cU(\bC) \mid q(\Im(\vec{y})) > 0\}, \]
where $\Im(\vec{y})$ is the imaginary part of the complex vector $\vec{y}$.
We also define the open cone:
\[ \indnotalpha{\Omega}{Omega} = \{ \vec{y} \in \cU(\bR)\mid q(\vec{y}) > 0 \}, \]
 as well as, the map $\Phi$ from $\cU(\bC) \rightarrow \cU(\bR)$ given by $\Phi(\vec{y}) = \Im(\vec{y})$ so that $\uhp_q = \Phi^{-1}(\Omega)$.
\end{df}

\begin{prop}
The map $\psi:\uhp_q\rightarrow\kappa$ given by $\psi(\vec{y}) \mapsto [-\tfrac{1}{2}(q(\vec{y})+q(e_2)),1,\vec{y})]$ is biholomorphic.
\end{prop}
This is a straight forward check (see \cite[Lem. 2.3.40]{Fiori_MSC}).

\begin{rmk}
The space $\uhp_q$ has $2$ components. 
To see this suppose $q$ has the form $q(x_1,...,x_n)= a_1x_1^2 - a_2x_2^2 -...-a_n x_n^2$ with $a_i>0$.
The condition imposed by $q(\Im(Z)) > 0$ gives us two components corresponding to $z_1 > 0$ and $z_1 < 0$. Under the map $\psi$ one of these corresponds to $\kappa^+$. We shall label that component $\uhp_q^+$.
\end{rmk}

Via the isomorphism with $\kappa$, we see that we have a transitive action of $\Orth_q^+(\bR)$ on $\uhp_q$.
One advantage to viewing the symmetric space under this interpretation is that it corresponds far more directly to some of the more classically constructed symmetric spaces such as the upper half plane.

\subsubsection{Conjugacy Classes of Morphisms $\bS\rightarrow \Orth_{2,n}$}

We now give the interpretation of the space as a Shimura variety.

We may (loosely) think of Shimura varieties as elements of a certain conjugacy classes of morphisms:
\[ h:(\indnotalpha{\bS}{Stor} = \Res_{\bC/\bR}(\bG_m)) \rightarrow \GO_{2,n} \]
satisfying additional axioms.
In particular, we are interested in those morphisms where the centralizer:
\[ Z_{\GO}(h(\bS)) =  Z(\GO_{2,n})\cdot K \simeq \bG_m\cdot(\Orth_2\times \Orth_n). \]
We get a bijection  between such maps and our space as follows:

  Given an element $\langle \vec{x},\vec{y}\rangle \in \Grass(V)$ we consider the morphism $h(re^{i\theta})$ defined by specifying that it acts as $\smtx {r^2\cos(2\theta)}{r^2\sin(2\theta)}{-r^2\sin(2\theta)}{r^2\cos(2\theta)}$ on the $\spann(\vec{x},\vec{y})$ and trivially on its orthogonal complement.

  Conversely, given $h$ in the conjugacy class of such a morphism we may take $[\vec{v}] \in \kappa^+$ to be the eigenspace of $r^2(\cos(2\theta)+i\sin(2\theta))$. 

The following claim is a straightforward check.
\begin{claim}
These two maps are inverses.
\end{claim}

Note that the two components correspond to swapping the (non-trivial) eigenspaces of $h$.

\subsubsection{Realization as a Bounded Domain}

For this section we will assume that:
      \[ \tilde{A} = \begin{pmatrix} 0 & 0 & 1 & 0 & \\
                         0 & 0 & 0 & 1 & \\ 
                         1 & 0 & 0 & 0 & \\ 
                         0 & 1 & 0 & 0 & \\ 
                           &   &   &   & A \end{pmatrix} \]
is the matrix for our quadratic form.  
This is not in general possible over $\bQ$ if $n\leq 4$. For the purpose of most of this discussion we work over $\bR$ and this fact is not a problem. However, it must be accounted for if ever rational structures are to be used.
In order to compute the bounded domain, we must work with the Lie algebra, and this is slightly easier if we change the basis using the matrix:
      \[ \begin{pmatrix} 1 & 0 & 1 & 0 & \\
                         0 & 1 & 0 & 1 & \\ 
                         1 & 0 & -1& 0 & \\ 
                         0 & 1 & 0 & -1& \\ 
                           &   &   &   & 1_{n-2} \end{pmatrix} \]
       so that the matrix for the quadratic form is:
      \[  \hat{A} = \begin{pmatrix} 2 & 0 & 0 & 0 & \\
                         0 & 2 & 0 & 0 & \\ 
                         0 & 0 & -2& 0 & \\ 
                         0 & 0 & 0 & -2& \\ 
                         &   &   &   & A \end{pmatrix} = \begin{pmatrix} 2 & 0 &  \\
                         0 & 2 & \\ 
                          &  & A' \\ 
                          \end{pmatrix}. \]
      We compute that the Lie algebra $\so_{\hat{A}}$ is $\mtx W{Z'}ZY$, 
      where $W \in \Mat_{2,2}$ is skew-symmetric, $Y \in \Mat_{n,n}$ is in $\so_{A'}$, $Z\in \Mat_{2,n}$, and $Z' = -Z^tA'/2$.
      We conclude that the eigenspaces for the action of the centre of $\fk$ on $\fp_\bC$ are $\fp_\pm$ are $\mtx 0{Z'}{Z}0$, where $Z=\begin{pmatrix} \vec{z}^t &\mp i\vec{z}^t \end{pmatrix}$ and $Z' = -Z^tA'/2$.

In order to compute the exponential of the Lie algebra we observe that the square of this matrix is equal to
\[ -\frac{1}{2}\mtx{Z^tA'Z}000 = 
                            -\frac{\vec{z}^tA'\vec{z}}{2}\begin{pmatrix} 1       & {\mp i} & \\ 
                                                                                                                     {\mp i} &  {-1} &  \\ 
                                                                                                                                 &         & 0 \end{pmatrix}, \]
       and that its cube is the zero matrix.
      We thus have that $P_\pm$ is   
     \[ \begin{pmatrix}  1_2 -\tfrac{1}{4}Z^tA'Z & -\tfrac{1}{2}Z^tA' \\
                                    Z                                      & 1_n \end{pmatrix}, \]
      where $Z=\begin{pmatrix} \vec{z}^t &\mp i\vec{z}^t \end{pmatrix}$.
      
      After undoing the change of basis $P_\pm$ becomes:
      \[ 1_{n+2} + 
         \frac{1}{2}\left(\begin{smallmatrix} 
                         0                  &-iz_1-z_2  & 2z_1        &-iz_1+z_2   & -\vec{z_3}A' \\
                        iz_1+z_2 &0                    &-iz_1-x_2  &2-iz_2          & i\vec{z_3}A'  \\ 
                         -2z_1         &iz_1-z_2    & 0                &-iz_1+z_2   &-\vec{z_3}A'  \\ 
                        iz_1-z_2 &2iz_2             &-iz_1-z_2 &0                      & i\vec{z_3}A'\\ 
                          \vec{z_3}^t & -i\vec{z_3}^t  & \vec{z_3}^t  &-i\vec{z_3}^t & 0 \end{smallmatrix}\right)
    - \frac{\vec{z}^tA'\vec{z}}{8}
         \left(\begin{smallmatrix} 1 &-i & 1 &-i & \\
                        -i &-1 &-i &-1 & \\ 
                         1 &-i & 1 &-i & \\ 
                        -i &-1 &-i &-1 & \\ 
                           &   &   &   & 0 \end{smallmatrix}\right),
                          \]
       where $\vec{z_3} = (z_3,z_4,\ldots,z_{n-2})$.
      The action of this matrix on $ \kappa^+$ takes $[1:i:1:i:\vec{0}]$ to:
      \[ \Psi(\vec{z}) =  [ (1,i,1,i,\vec{0}) + 2(z_1,z_2,-z_1,-z_2,\vec{z_3}) - \tfrac{1}{2}\vec{z}^tA'\vec{z}(1,-i,1,-i,\vec{0}) ] \in N. \]
       One may check that this is an injective map.
We thus conclude that $\cD$ is the bounded domain:
\[ \{  (z_1,z_2,\vec{z_3}) \subset P_+ | \text{ conditions }  \}. \]
   The conditions are computed by pulling them back from $P(V(\bC))$. The resulting conditions can be expressed as:
\begin{align*}
4 + 4\vec{z}A'\overline{\vec{z}^t} + \abs{\vec{z}A'\vec{z}^t}^2 &> 0 \text{ and }\\
4- \abs{ \vec{z}A'\vec{z}^t}^2 &> 0. 
\end{align*}

We have the following maps between these models:
\begin{align*}
 \Psi:\text{Bounded}  &\rightarrow \text{Projective} \\
 \Psi^{-1} : \text{Projective} &\rightarrow \text{Bounded} \\
  \Upsilon: \text{Bounded}  &\rightarrow \text{Tube Domain} \\
   \Upsilon^{-1}: \text{Tube Domain}  &\rightarrow \text{Bounded}
 \end{align*}
The definition of the map $\Psi$ is implicit in the above computations.

Set $s(\vec{z}) = 1-2z_1-\tfrac{1}{2}\vec{z}A'\vec{z}^t$ then $\Upsilon$ is defined by:
\begin{align*}
      y_1 &=   \frac{i+2z_2+i\vec{z}A'\vec{z}^t }{s(\vec{z})}, \\
      y_2 &=   \frac{i-2z_2+i\vec{z}A'\vec{z}^t }{s(\vec{z})} \text{ and}\\
      y_{i} &= \frac{2z_i}{s(\vec{z})} \text{ for }i>2.
\end{align*}

To define an inverse to $\Upsilon$ set:
\[
\vec{y}' = ( \tfrac{1}{4}(iy_1+iy_2 + \vec{y}A''\vec{y}^t), \tfrac{1}{4}(y_1-y_2),
                    - \tfrac{1}{4}(iy_1+iy_2 + \vec{y}A''\vec{y}^t), -\tfrac{1}{4}(y_1-y_2), \vec{y}_3 ).
\]
Now set:
\[ r(\vec{y}) = \frac{ \vec{y}A''\vec{y}^t}{(\vec{y}')A''(\vec{y}')^t}.\]
Notice that $r(\Upsilon(\vec{z})) = 1-2z_1 - \tfrac{1}{2}\vec{z}A'\vec{z}^t$.
We can therefore define $\Upsilon^{-1}$ via:
\begin{align*}
        z_1 &= \tfrac{1}{4}r(\vec{y})(\vec{y}A'\vec{y}^t  + i(y_1+y_2)) + 1,\\
        z_2 &= \tfrac{1}{4}r(\vec{y})(y_1-y_2) \text{ and}\\
        z_i  &= r(\vec{y})y_i   \text{ for } i>2.
\end{align*}

\subsection{Boundary Components and the Minimal Compactification}

Locally symmetric spaces are often non-compact. It is thus often useful while studying them to construct compactifications.
We present here some of the most basic notions of this very rich theory.
For more details see \cite{Helgason, BorelJi, AMRT, Namikawa}.

\begin{df}
Consider a Hermitian symmetric domain $\cD$ realized as a bounded domain in $P_+$. 
We say $x,y \in \overline{\cD}$ are in the same \inddefss{boundary component}{symmetric space} if there exist maps:
\[ \varphi_j : \bH \rightarrow \overline{\cD} \quad j=1,\ldots,m \]
with $\varphi_j(\bH) \cap \varphi_{j+1}(\bH)\neq \emptyset$, and
there exist $x',y'\in \bH$ such that $\varphi_1(x') = x$ and $\varphi_m(y')=y$.

We say that two boundary components $F_1,F_2$ are adjacent if $\overline{F_1}\cap\overline{F_2}\neq 0$.
\end{df}

\begin{thm}\label{thm:BoundaryComponents}
The boundary components of the Hermitian symmetric domain $\cD$ are the maximal sub-Hermitian symmetric domains in $\overline{\cD}$. Moreover, they satisfy the following:
\begin{itemize}
\item
The group $G$ acts on boundary components preserving adjacency.
\item The closure $\overline{\cD}$ can be decomposed as $\overline{\cD} = \sqcup_\alpha \indnota{F_\alpha}$, where the $F_\alpha$ are boundary components.
\item For each boundary component $F_\alpha$ there exists a map:
\[ \varphi_\alpha : \SL_2(\bR) \rightarrow G \]
inducing a map
\[ f_\alpha : \overline{\uhp} \rightarrow \overline{\cD} \]
such that $f_\alpha(i) = o$ (for the fixed base point $o=K$) and $f_\alpha(i\infty) \in F_\alpha$.
\end{itemize}
\end{thm}
See \cite[Thm. 1,2 Sec 3.3]{AMRT}.

\begin{thm}
There is a bijective correspondence between the collection $\{ F_\alpha \}$ of boundary components and the collection of real ``maximal" parabolic subgroups $\indnota{P_\alpha}$ of $G = \Aut(\cD)$. (By ``maximal" we mean that for each simple factor $G_i$ of $G$ the restriction to the factor is either maximal or equal to $G_i$).

Explicitly we have $P_\alpha = \{ g\in G \mid g F_\alpha = F_\alpha \}$. Moreover, $F_\alpha \subset \overline{F}_\beta$ if and only if $P_\alpha \cap P_\beta$ is a parabolic subgroup.
\end{thm}
See \cite[Prop. 1,2 Sec 3.3]{AMRT}.

\begin{df}
We say $F_\alpha$ is a \inddefs{rational}{boundary component} boundary component if $P_\alpha$ is defined over $\bQ$.
We define the space:
\[  \cD^\ast = \underset{rational}\cup F_\alpha. \]
\end{df}

\begin{thm}
Let $\Gamma\subset G(\bQ)$ be an arithmetic subgroup.
There exists a topology on $\cD^\ast$ such that the quotient $\overline{X}^{\textrm{Sat}} := \Gamma\backslash \cD^\ast$ has the structure of a normal analytic space.

We call $\indnotalpha{\overline{X}^{\textrm{Sat}}}{Xsat}$ the minimal Satak\'e compactification of $\Gamma\backslash \cD$.
\end{thm}
See \cite[Sec. III.3]{BorelJi}.

\begin{rmk}
The topology one should assign may become more apparent once we introduce other compactifications.
\end{rmk}

\section{Modular Forms}
\label{sec:SectionModular}

We give now a simplified notion of modular forms. More general and precise definitions can be found in any of \cite{BorelAutomorphic,Mumford_Proportionality,BailyBorel}.
\begin{df}
Let $\cQ$ be the image of $\cD=G/K$ in the projective space $\breve{\cD}= G_\bC/P^-$ and let $\tilde{\cQ}$ be the cone over $\cQ$.
A \inddef{modular form} $f$ for $\Gamma$ of weight $k$ on $\cD$ can be thought of as any of the equivalent notions:
\begin{enumerate}
\item A function on $\tilde{\cQ}$ homogeneous of degree $-k$ which is invariant under the action of $\Gamma$.

\item A section of $\Gamma\backslash(\cO_{\breve{\cD}}(-k)|_\cD)$ on $\Gamma\backslash \cD$.

\item A function on $\cQ$ which transforms with respect to the $k^{th}$ power of the factor of automorphy under $\Gamma$.
\end{enumerate}
To be a \inddefs{meromorphic}{modular form} (resp. \inddefs{holomorphic}{modular form}) modular form we require that $f$ extends to the boundary and that it be meromorphic (resp. holomorphic). 
One may also consider forms which are holomorphic on the space but are only meromorphic on the boundary.
\end{df}
\begin{rmk}
The condition at the boundary depends on understanding the topology, a concept we have not yet defined.
There is an alternative definition in terms of Fourier series.
Let $\cU_\alpha$ be the centre of the unipotent radical of $P_\alpha$ and set $U_\alpha = \Gamma \cap \cU_\alpha$. This group is isomorphic to $\bZ^m$ for some $m$ and the function $f$ is invariant under its action.
The boundary condition can be expressed by saying the non-trivial Fourier coefficients  (which are indexed by elements of $U_\alpha^\ast$), are contained in a certain self-adjoint cone $\Omega_\alpha \subset \cU_\alpha^\ast$.
\end{rmk}

The following is what is known as the Koecher principle (see for example \cite{Freitag_HMF}).
\begin{claim}
If the codimension of all of the boundary components is at least $2$, then every form which is holomorphic on $\cD$ extends to the boundary as a holomorphic modular form.
\end{claim}
This result is a consequence of results about extending functions on normal varieties.

\begin{thm}[Baily-Borel]
Let $M(\Gamma,\cD)$ be the graded ring of modular forms then 
\[ \indnotalpha{\overline{X}^{BB}}{Xbb} := \Proj(M(\Gamma,\cD)) \]
is the Baily-Borel compactification of $X$.
Moreover, this is isomorphic to the minimal Satak\'e compactification as an analytic space
\end{thm}
See \cite{BailyBorel} and \cite[III.4]{BorelJi}. 

\subsection{The O(2,n) Case}

Specializing the previous section to the orthogonal case we can use the following definition for modular forms.

\begin{df}
Let $\overline{\kappa}^+ = \{\vec{v}\in V(\bC) \mid [\vec{v}]\in \kappa^+\}$ be the cone over $\kappa^+$.
Let $k\in\bZ$, and $\chi$ be a character of $\Gamma$. A meromorphic function on $\overline{\kappa}^+$ is a \inddef{modular form} of weight $k$ and character $\chi$ for the group $\Gamma$ if it satisfies the following:
\begin{enumerate}
\item $F$ is homogeneous of degree $-k$, that is, $F(c\vec{v}) = c^{-k}F(\vec{v})$ for $c\in\bC-\{0\}$.
\item $F$ is invariant under $\Gamma$, that is, $F(g\vec{v})=\chi(g)F(\vec{v})$ for any $g\in\Gamma$.
\item $F$ is meromorphic on the boundary.
\end{enumerate}
If $F$ is holomorphic on $\overline{\kappa}^+$ and on the boundary then we call $F$ a holomorphic modular form.
In this case $\cU_\alpha$ and $\Omega_\alpha$ are precisely those introduced for the tube domain model (see Section \ref{sec:SectionTubeDomain}).
\end{df}
\smallskip

\begin{rmk}
The Koecher principle implies condition (3) is automatic if the dimension of maximal isotropic subspace is less than $n$. Noting that for type $(2,n)$ the Witt rank is always at most $2$, we see that the Koecher principle often applies.
\end{rmk}

\begin{rmk}
One of the best sources of examples of modular forms for these orthogonal spaces is the Borcherds lift (see \cite{Borcherd_inv, Brunier_infprod, Brunier_BP} for more details).
The Borcherds lift, which may be defined via a regularized theta integral, takes nearly holomorphic vector-valued modular forms for the upper half plane and constructs modular forms on an orthogonal space.
The forms constructed this way have well understood weights, levels, and divisors.
One can also consider other types of forms (for example Eisenstein series, Poincare series and theta series).
\end{rmk}

}

{\def\XMetaCompile{1}

\section{Toroidal Compactifications}
\label{sec:SectionToroidal}

We will now introduce the notion of toroidal compactifications. Many more detailed references exist (see for example \cite{AMRT,  Pink_Dissertation, FaltingsChai, LAN_PHD, PERA_PHD, Namikawa}).
Toroidal compactifications play an important role in giving geometric descriptions of modular forms,  as well as in computing dimension formulas
 (see Section \ref{sec:SectionDimension}).

The key idea of toroidal compactifications of locally symmetric space is that locally in a neighbourhood of the cusps, the space looks like the product of an algebraic torus and a compact space. We thus compactify locally at the cusp by compactifying the torus. Doing this systematically allows us to glue the parts together to get the compactification we seek.

\subsection{Torus Embeddings}

We give a very brief overview of toric varieties. For more details see
\cite{FultonToric, KKMD_Toroidal, Oda_Toroidal, AMRT}.
For the purpose of this section we will restrict our attention to complex tori though many results hold in greater generality.

\begin{df}
By a torus $T$ over $\bC$ of rank $n$ we mean an algebraic group isomorphic to $\bG_m^n$ so that $T(\bC)= (\bC^\times)^n$.
We  shall denote  the characters and cocharacters of $T$ by $\indnota{X^\ast(T)}$ and $\indnota{X_\ast(T)}$.
There exists a pairing between $X^\ast(T)$ and $X_\ast(T)$
\[   \Hom(T,\bG_m) \times \Hom(\bG_m,T) \rightarrow \Hom(\bG_m,\bG_m) \simeq \bZ \]
given by $(f,g)\mapsto f\circ g$.
\end{df}

We have the following basic results:
\begin{itemize}
\item $X^\ast(T) \simeq X_\ast(T) = \bZ^n$
\item $\Lie(T) \simeq \bC^n$ with the trivial bracket.
\item We have an exact sequence:
\[ \xymatrix{
 0 \ar@{->}[r]& X_\ast(T) \ar@{->}[r] \ar@{=}[d]& \Lie(T)\ar@{->}[r]\ar@{=}[d] & T  \ar@{->}[r]\ar@{=}[d] &0 \\
 0 \ar@{->}[r]& \bZ^n  \ar@{->}[r]  & \bC^n  \ar@{->}[r]^{exp} &  (\bC^\times)^n  \ar@{->}[r]  & 0}
\]
\item As an algebraic variety $T = \Spec(\bC[X^\ast(T)])$.
\end{itemize}

\begin{ex}
Before proceeding let us give a few basic examples of compactifications of tori.
\begin{itemize}
\item Compactification of $\bC^\times$.

We have $\bC^\times \injects \bP^1$ via $x\mapsto [x:1]$.
The closure then contains $[0:1]$ and $[1:0]$.

\item Compactification of $(\bC^\times)^2$.
We may consider maps $(\bC^\times)^2 \rightarrow \bP^2$ or $(\bC^\times)^2 \rightarrow \bP^1\times \bP^1$ given respectively by:
\[ (x,y)\mapsto [x:y:1] \text{ and }(x,y)\mapsto ([x:1],[y:1])  \]
In the first case the boundary is $3$ copies of $\bP^1$ ($[0:y:1],[x:0:1],[1:y:0]$) in the second it is $4$ ($([x:y],[0:1]),([x:y],[1:0]),([0:1],[x:y]),([1:0],[x:y])$).
Notice that in both cases the copies of $\bP^1$ we have added form a chain with intersections at $0,\infty$.
\end{itemize}
We notice that in all these examples the torus $T$ acts on its compactification and we have a natural orbit decomposition.
\end{ex}

\begin{qu*}
Can these types of embeddings be characterized systematically?
\end{qu*}
The answer is given by the following definition:
\begin{df}
A \inddef{torus embedding} consists of a torus $T$ with a Zariski open dense embedding into a variety $X$ together with an action of $T$ on $X$ which restricts to the group action on the image of $T$ in $X$.

A morphism of torus embeddings $(T,X)\rightarrow (T',X')$ consists of a surjective morphism $f:T\rightarrow T'$ and an equivariant morphism $\tilde{f}:X\rightarrow X'$ extending $f$.
\end{df}

\begin{qu*}
How can one describe an inclusion of $T$ into another space?
\end{qu*}
The answer is given by the following claim.
\begin{claim}
A map from a torus $T$ into an affine variety $X$ can be constructed by considering any submonoid $M\subset X^\ast(T)$ and the map $T = \Spec(\bC[X^\ast(T)]) \rightarrow \Spec(\bC[M]) = X$ induced by the inclusion of $M\injects X^\ast(T)$.
\end{claim}

The above suggests an approach to the problem, we now proceed to make it systematic.

\subsubsection{Cones and Cone Decompositions}

If one works with the idea it becomes apparent that a random submonoid will lead to a poorly structured variety.
As such we are interested in defining `nice' submonoids that will lead to `nice' varieties.

\begin{df}
Let $N_\bR$ be a real vector space, a \inddef{cone} $\Omega\subset N_\bR$ is a subset such that $\bR_+\cdot \Omega = \Omega$.

$\Omega$ is said to be \inddefs{non-degenerate}{cone} if $\overline{\Omega}$ contains no straight lines.

$\Omega$ is \inddefs{polyhedral}{cone} if there exists $x_1,\ldots,x_n \in N_\bR$ such that $\Omega = \{ \sum_i a_ix_i \mid a_i \in\bR^+\cup \{0\} \}$.

$\Omega$ is \inddefs{homogeneous}{cone} if $\Aut(\Omega,N_\bR)$ acts transitively on $\Omega$.

The \inddefs{dual}{cone}  of $\overline{\Omega}$ is $\indnotalpha{\overline{\Omega}^\vee}{Omegavee} = \{ v^\vee\in N_\bR^\vee\mid v^\vee(y) \ge 0\text{ for all } y\in \Omega\}$. The dual of $\Omega$ is the interior of $\overline{\Omega}^\vee$.

We say $\Omega$ is \inddefs{self-adjoint}{cone} (with respect to $\langle\cdot,\cdot\rangle$) if there exists a positive-definite form on $N_\bR$ whose induced isomorphism $N_\bR\simeq N_\bR^\vee$ takes $\Omega$ to $\Omega^\vee$.
\end{df}

\begin{rmk}
Polyhedral cones are by definition closed, whereas homogeneous cones are relatively open.
\end{rmk}

\begin{ex}
The first $5$ examples cover all the examples of simple open homogeneous self-adjoint cones.
\begin{itemize}
\item In $\bR^n$ the cone $\{ (x_1,\ldots,x_n) \mid x_1^2 - \sum_{i>1} x_i^2 > 0 \text{ and }x_1>0 \}$.

\item The cone of positive-definite matrices in $M_n(\bR)$.

\item The cone of positive-definite Hermitian matrices in $M_n(\bC)$.

\item The cone of positive-definite quaternionic matrices in $M_n(\bH)$.

\item The cone of positive-definite octonionic matrices in $M_3(\bO)$.

\item A more general version of the first case which we shall use in the sequel is the following.
Consider the quadratic space whose bilinear form $b$ is $\smtx 0110 \oplus (-A)$ where $A$ gives a positive-definite quadratic form.
Set 
\[\Omega = \{ \vec{v}\in N_\bR \mid b( \vec{v}, \vec{v}) > 0 ,\; v_1 > 0 \}.\]
The cone is open, non-degenerate and convex.
It is also self adjoint with respect to the inner product $x_1^2+x_2^2+\vec{x}_3A\vec{x}^t_3$.
We claim $\Omega^\vee$ is given by:
\[ \{ (a_1,a_2,\vec{a_3}) \mid 2x_1x_2 > \vec{x_3}^tA\vec{x} \Rightarrow  a_1x_1+a_2x_2+\sum_{i>3} a_iAx_i > 0 \}. \]
Indeed, by rescaling, suppose $x_2=1$. Then we have:
\[
   2x_1x_2 > \vec{x_3}^tA\vec{x_3} \Rightarrow a_1x_1+a_2+\sum_{i>3} a_iAx_i > a_1\tfrac{1}{2}\vec{x_3}^tA\vec{x}+a_2+\sum_{i>3} a_iAx_i. \]
Now writing: 
\[ a_1\tfrac{1}{2}\vec{x_3}^tA\vec{x_3}+a_2+\sum_{i>3} a_iAx_i = \tfrac{a_1}{2}(\vec{x_3}+\frac{\vec{a_3}}{a_1})^tA(\vec{x_3}+\frac{\vec{a_3}}{a_1}) + a_2-\frac{1}{2a_1}\vec{a_3}^tA\vec{a_3} \]
we see that this is larger than zero provided that $2a_1a_2 > \vec{a}_3^tA\vec{a}_3$.
In particular if $\vec{a}\in \Omega$.
\end{itemize}
\end{ex} 

\begin{df}
Given a subset $\Omega$ of a real vector space $N_\bR$ we say a set $\Sigma = \{ \sigma_i \}$  is a \inddefs{convex polyhedral}{cone decomposition} decomposition of $\Omega$ if:
\begin{itemize}
\item $\Omega = \cup \sigma_i$.
\item The $\sigma_i$ are convex polyhedral cones.
\item $\sigma_i\cap \sigma_j = \sigma_k \in \Sigma$ is a face of both $\sigma_i$ and $\sigma_j$.
\end{itemize}
We may also refer to the decomposition as a \inddefs{partial convex polyhedral}{cone decomposition} decomposition of $N_\bR$.

We make partial convex polyhedral decompositions into a category by requiring morphisms be of the following form.
Given decompositions $\indnotalpha{\Sigma_N}{SigmaN},\Sigma_M$ of $N_\bR,M_\bR$ ,respectively, a morphism is a linear map $f:N_\bR\rightarrow M_\bR$ such that
for all $\sigma_N \in \Sigma_N$ there exists $\sigma_M\in \Sigma_M$ such that $f(\sigma_N) \subset \sigma_M$.
\end{df}
\begin{rmk}
Note that the definition requires that $0\in \Omega$, thus $\Omega$ can not be both open and non-degenerate.
The space $N_\bR$ in which we are interested will almost always be either $X^\ast(T) \times \bR$ or $X_\ast(T) \times \bR$.
Typically the space $\Omega$ we consider are either all of $N_\bR$ or the rational closure $\indnotalpha{\overline{\Omega'}^{\textrm{rat}}}{Omegarat}$ (the convex hull of the rational rays in $\overline{\Omega'}$) where $\Omega'$ is an open homogeneous self adjoint cone.
\end{rmk}

\begin{df}
Given a cone decomposition $\Sigma$ of $\Omega$ we define a space $\indnota{N_\Sigma}$ as follows:
\[ N_\Sigma = \{ \indnota{y + \infty \sigma} \mid y \in N/\spann{\sigma}, \sigma \in \Sigma \}. \]

We put a topology on this by specifying when limits converge.  We say
\[ \lim y_n + \infty\sigma = x + \infty\tau \]
for $\sigma$ a face of $\tau$ if
\begin{enumerate}
\item $\lim y_n + \infty\tau = x+\infty \tau$ and 
\item for any splitting $\spann{\tau} = \spann{\Omega_\sigma} \oplus L'$ and for any $z\in \sigma$ we have $\rho(y_n) \in \sigma + z$ for all sufficiently large  $n$.
\end{enumerate}

We denote by $\Omega_\Sigma$ the correspondingly enlarged object.
\end{df}

\subsubsection{Constructing Torus Embeddings from Cone Decompositions}

\begin{df}
Given a torus $T$ and a convex polyhedral cone $\sigma \subset X_\ast(T)\otimes \bR$ we define a variety $X_\sigma$ as follows:
\[ \indnota{X_\sigma} = \Spec( k[ X^\ast(T) \cap \sigma^\vee ] ). \]
This variety comes equipped with a map $T \rightarrow X$ arising from the inclusion:
\[ k[ X^\ast(T) \cap \sigma^\vee] \injects k[X^\ast(T)]. \]
\end{df}

\begin{df}
Given a torus $T$, a cone $\Omega \subset X_\ast(T)\otimes \bR$, and a convex polyhedral cone decomposition $\Sigma$ of $\Omega$, we define a variety
$\indnota{X_\Sigma}$ as follows. It has an open cover by affines:
\[ X_\sigma = \Spec( k[ X^\ast(T) \cap \sigma^\vee] ) \]
for each $\sigma,\tau \in \Sigma$.
We glue $X_\sigma$ and  $X_\tau$ along their intersection $X_\sigma \cap X_\tau = X_{\tau\cap\sigma}$.
\end{df}

\begin{prop}
There is an action of $T$ on $X_\Sigma$. Moreover, there is a bijection between the orbits of $T$ in $X_\Sigma$ and $\Sigma$.
We express this bijection by writing $\indnota{O(\sigma)}$ for an orbit of $T$.
Moreover, there is a continuous map  $\Im : X_{\bC,\Sigma} \rightarrow \Omega_\Sigma$.
It maps the orbit $O(\sigma)$ to $X_\ast(T)\otimes_\bR + \infty\sigma$.
\end{prop}
See \cite[Thm. 4.2]{Oda_Toroidal}.

\subsubsection{Properties of Torus Embeddings}

We now summarize a number of geometric results concerning torus embeddings.
For more details see \cite{Oda_Toroidal}.

\begin{df}
We say a convex polyhedral cone $\sigma$ is \inddefs{rational}{cone decomposition} (with respect to an integral structure $N_\bZ$ in the ambient space) if there exists $r_1,\ldots,r_m \in N_\bZ$ such that:
\[ \sigma = \{ x \mid \langle r_i,x \rangle \ge 0 \text{ for all } i \}.\]
For $f$ to be a morphism of rational partial polyhedral cone decompositions we require that $f(N_\bZ) \subset M_\bZ$ and that $M_\bZ/f(N_\bZ)$ be finite.
\end{df}

\begin{thm}
There exists an equivalence of categories between normal separated locally of finite type torus embeddings and rational partial polyhedral decompositions.
Moreover, the variety is finite type if and only if $\Sigma$ is finite (as a set).
\end{thm}
See \cite[Thm. 4.1]{Oda_Toroidal}.
\begin{rmk}
One could obtain torus embeddings which are not normal by using monoids which do not arise from cones, and which are not separated by using cones whose intersections contain open subsets of both. One can obtain non-locally of finite type torus embeddings by removing the requirement of rationality.
\end{rmk}

For the remainder of this section assume all torus embeddings are normal separated and  locally of finite type.

\begin{thm}
A morphism of torus embeddings is proper if and only if the associated morphism of cone decompositions is surjective and the preimage of every cone is finite.

Consequently, a torus embedding is complete if and only if it is finite and decomposes all of $X_\ast(T)\otimes \bR$.
\end{thm}
See \cite[Thm 4.4]{Oda_Toroidal}.
\begin{rmk}
In particular finite refinements are proper.
\end{rmk}

\begin{df}
A rational convex cone $\sigma$ is said to be \inddefs{regular}{cone} if it has a generating set which is a basis for its span, that is, there exists $x_1,\ldots,x_n$ such that:
\[ \sigma = \{ \sum a_ix_i \mid a_i \in \bR^+ \} \]
and $x_i$ form a basis for $N_\bZ\cap \spann{\sigma}$.
\end{df}

\begin{prop}
A torus embedding is regular if and only if all of its cones are regular.
\end{prop}
See \cite[Thm. 4.3]{Oda_Toroidal}

\begin{df}
A convex rational polyhedral cone decomposition $\Sigma$ is said to be \inddefs{projective}{cone decomposition} if there exists a continuous convex piecewise linear function $\phi: V \rightarrow \bR$ such that the following properties hold:
\begin{enumerate}
\item $\phi(x) > 0$ for $x\neq 0$.
\item $\phi$ is integral on $N_\bZ$.
\item The top dimensional cones $\sigma$ are the maximal polyhedral cones in $\overline{\Omega}$ on which $\phi$ is linear.
\end{enumerate}
\end{df}

\begin{thm}[Projectivity]
If $\Sigma$ is projective then the torus embedding corresponding to $\Sigma$ is quasi-projective.
\end{thm}
The statement of the result is \cite[Prop. 6.14]{Namikawa}. For the proof see \cite[Sec. 6]{Oda_Toroidal}.

\begin{thm}
Let $X_\Sigma$ be a torus embedding of finite type, then there exists a refinement $\Sigma'$ of $\Sigma$ such that $X_{\Sigma'}$ is non-singular and $X_{\Sigma'}$ is the normalization of a blowup of $X_\Sigma$ along an ideal sheaf.
\end{thm}
See \cite[Thm. 10,11]{KKMD_Toroidal}.
Concretely one may use iterated barycentric subdivisions to find such a refinement.

\subsection{Toroidal Compactifications}

We now describe how to construct the toroidal compactification for a general Hermitian symmetric domain of the non-compact type.
The proofs that the constructions we describe have the desired properties can be found in \cite{AMRT}. 
In practice one is able to explicitly compute all the objects involved.
 See for example \cite{Namikawa} for the Siegel case, or the following section for the orthogonal case.
What is in fact much harder is describing explicitly a good choice of cone decomposition and the resulting space.

The key objects involved in the construction are the following:
\begin{itemize}
\item A Hermitian locally symmetric domain of the non-compact type:
\[ X = \Gamma \backslash \cD = \Gamma\backslash G/K. \]
  
\item The maximal (real) parabolic subgroups $P_\alpha \leftrightarrow F_\alpha$ which correspond to boundary components:
\[ P_\alpha = \{ g\in G \mid gF_\alpha = F_\alpha \}. \]

\item 
The unipotent radical $\indnotalpha{\cW_\alpha}{Wcalpha}$ of $P_\alpha$.

\item 
The centre $\indnotalpha{\cU_\alpha}{Ucalpha} = Z(\cW_\alpha)$ and the quotient $\indnotalpha{\cV_\alpha}{Vcalpha} = \cW_\alpha/\cU_\alpha$.

\item 
An open self adjoint homogeneous cone $\indnotalpha{\Omega_\alpha}{Omegaalpha} \subset \cU_\alpha$.
It is the orbit under conjugation of $P_\alpha$ acting on $\varphi(\smtx 1101) \in \cU_\alpha$ where $\varphi$ is the map of Theorem \ref{thm:BoundaryComponents}.
(See \cite[Sec 3.4.2]{AMRT}).

\item 
The pieces of the Levi decomposition $\indnota{G_{h,\alpha}}, \indnota{G_{\ell,\alpha}}$ of $P_\alpha$. They are characterized by the fact that
\[ G_{h,\alpha} = \Aut(F_\alpha) \text{ and } G_{\ell,\alpha} = \Aut(\Omega_\alpha,\cU_\alpha). \]
These morphisms are realized by maps $p_{h,\alpha},p_{\ell,\alpha}$.

\item 
A decomposition $P_+ \supset \indnotalpha{\cB_\alpha}{Bcalpha} = \cU_\alpha\cD = F_\alpha \times\cV_\alpha \times \cU_{\alpha,\bC}$. The inclusion  $\cD \injects \cB_\alpha$ realizes  $\cD$ as a fibre bundle of cones and vector spaces.
The natural projections are equivariant for the actions of $G_{h,\alpha}$ on $F_\alpha$ and $G_{\ell,\alpha}$ on $\cU_{\alpha,\bC}$ through $\indnota{p_{h,\alpha}}$ and $\indnota{p_{\ell,\alpha}}$.
This map has an intrinsic description, see \cite[Sec. 3.3.4]{AMRT}.

\item A map $\indnotalpha{\Phi_\alpha}{Phialpha} : \cB_\alpha \rightarrow \cU_\alpha$ such that $\cD = \Phi_\alpha^{-1}(\Omega_\alpha)$.
This map is equivariant for the actions of $G_{\ell,\alpha}$.
This map has an intrinsic description, see \cite[Sec. 3.4]{AMRT}.
\end{itemize}

We make the following additional definitions:
\begin{itemize}
\item $\indnotalpha{\Gamma_\alpha}{Gammaalpha} = \Gamma \cap P_\alpha$.
\item $\indnota{W_\alpha} =\Gamma_\alpha \cap \cW_\alpha$.
\item $\indnota{U_\alpha} = W_\alpha \cap \cU_\alpha$.
\item $\indnota{V_\alpha} = W_\alpha/U_\alpha$.
\end{itemize}

In order to compactify the space we shall need to describe the space locally using the following collection of open covers.
\[ \xymatrix{ 
             \Omega_\alpha  &   \subset    & i\cU_\alpha = iX_\ast(T)\otimes \bR \\
            U_\alpha\backslash  \cD \subset  U_\alpha \backslash \cB_\alpha \ar[d]& = & F_\alpha\times \cV_\alpha\times \cU_{\alpha,\bC}/U_\alpha \ar[u]^\Phi \ar[d]^{\pi^i} \\
            \cU_{\alpha,\bC} \backslash \cB_\alpha & = & F_\alpha\times  V_\alpha }
            \]
We define $T_\alpha = \cU_{\alpha,\bC}/U_\alpha$. It is an algebraic torus (over $\bC$).

Heuristically one may think of $\Omega_{\alpha}$ as $(\mathbb{R}^+)^n$ and
\[
U_{\alpha} \backslash \cD = 
\{ ( \tau_1, \tau_2, \tau_3)  \mid 0 < \abs{ \overline{\tau}_3 } \}.
\]
It thus seems natural to add points for $\overline{\tau}_3 =0.$ The Satak\'e compactification effectively adds $(\tau_1,\tau_2, \tau_3)\in F_\alpha\times (0) \times (0)$.
The collection of points added for the toroidal compactifications we are considering shall typically be larger.
In order to functorially control the set of points added, we shall need the auxiliary information of cone decompositions.

\begin{df}
A \inddefsalpha{$\rho_{\ell,\alpha} (\Gamma_{\alpha})$-admissible}{cone decomposition}{rholalphagamma-admissible polyhedral} polyhedral decomposition of $\Omega_{\alpha}$ is $\Sigma_{\alpha}= \{ \sigma_{\nu} \}$ 
(relative to the $U_{\alpha}$ rational structure on $\cU_{\alpha}$)
such that:
\begin{enumerate}

\item $\Sigma_{\alpha}$ is a rational convex polyhedral cone decomposition of $\overline{\Omega_{\alpha}}^{\textrm{rat}}$ the rational closure of $\Omega_{\alpha}$,

\item $\Sigma_{\alpha}$ is closed under the action of $\rho_{\ell,\alpha}(\Gamma_{\alpha})$, and

\item only finitely many $\rho_{\ell,\alpha} (\Gamma_{\alpha})$ orbits in $\Sigma_{\alpha}$.

\end{enumerate}
A \inddefsalpha{$\Gamma$-admissible family}{cone decomposition}{gamma-admissible family} of polyhedral decompositions is 
$\Sigma=\{ \Sigma_{\alpha}\}_{\alpha \textrm{ Rational}}$ where:

\begin{enumerate}

\item $\Sigma_{\alpha}$ is a $\rho_{\ell,\alpha} (\Gamma_{\alpha})$-admissible polyhedral decomposition,

\item for $\gamma \in \Gamma$ if $\gamma F_{\alpha}=F_{\beta}$ then $\gamma \Sigma_{\alpha} \gamma ^{-1} = \Sigma_{\beta}$, and

\item for $F_{\alpha}$ a boundary component of $F_{\beta}$ then $\Sigma_{\beta} =\Sigma_{\alpha} \cap \mathcal{U}_{\beta}.$

\end{enumerate}
\end{df}

\subsubsection{The points being added}

Now given such a $\rho_{\ell,\alpha} (\Gamma_{\alpha})$-admissible polyhedral decomposition $\Sigma_\alpha$ we may construct:
\begin{align*}
\left( U_{\alpha} \backslash \cB_\alpha\right)_{\Sigma_{\alpha}}&=
F_\alpha\times \cV_\alpha \times (T_{\alpha})_{\Sigma_{\alpha}} \\
&= \sqcup_{\sigma \in \Sigma_{\alpha}} O(\sigma),
\end{align*}
where $O(\sigma)=F_\alpha\times \cV_\alpha \times O'(\sigma)$ and $O'(\sigma)$ is the points added with respect to $\sigma$ to $T_\alpha$. The $O(\sigma)$ have the following properties:
\begin{enumerate}

\item $O(\sigma)$  a torus bundle  over $F_\alpha\times \cV_\alpha$,

\item for $\sigma < \tau$ have that $O(\tau) \subset \overline{O(\sigma)} $ (for $\sigma=\{0\}$ we have $O(\sigma)=U_{\alpha} \backslash \cB_\alpha$), and

\item dim $\sigma +$ dim $O(\sigma)= $ dim $\mathcal{D} .$

\end{enumerate}

There is a map $\Im$:
\begin{align*}
\Im: F_{\alpha} \times \cV_{\alpha} \times \cU_{\alpha, \bC}
&\rightarrow \cU_{\alpha}\\
(\tau_1, \tau_2, \tau_3)
&\mapsto \Im (\tau_3)
\end{align*}
which projects onto the imaginary part of $\cU_\alpha$. It can be extended/descended to a map
\[
\Im:(U_{\alpha} \backslash \cB_{\alpha})_{\Sigma_{\alpha}} \rightarrow (\cU_{\alpha})_{\Sigma_{\alpha}}
\]
in such a way so that:
\[
\Im: O(\sigma) \rightarrow \{ y + \infty\sigma \}.
\]
Using the fact that $\Phi_\alpha$ is a translation of $\Im$ we can then extend $\Phi_\alpha$ to:
\[
\Phi_\alpha:(U_{\alpha} \backslash \cB_{\alpha})_{\Sigma_{\alpha}} 
\rightarrow (\cU_{\alpha})_{\Sigma_{\alpha}}.
\]
With all this in hand, we make the following definitions:
\begin{align*}
(U_{\alpha} \backslash \cD)_{\Sigma_{\alpha}} 
&:=\textrm{ Interior of closure of } U_{\alpha} \backslash \cD \textrm{ in } (U_{\alpha} \backslash \cB_\alpha)_{\Sigma_{\alpha}}, \text{ and} \\
(\Omega_{\alpha})_{\Sigma_{\alpha}} 
&:=\textrm{ Interior of closure of }\Omega_{\alpha} \textrm{ in } (\cU_{\alpha} )_{\Sigma_{\alpha}}.
\end{align*}
By continuity if follows that $\Phi^{-1}((\Omega_{\alpha})_{\Sigma_{\alpha}})=( U_{\alpha} \backslash \mathcal{D} )_{\Sigma_{\alpha}}.$ This observation allows one to check the following claim:

\begin{claim}
If $\sigma \cap \Omega_{\alpha} \neq \varnothing$ then $O(\sigma) \subset (U_{\alpha} \backslash \cD)_{\Sigma_{\alpha}} .$
\end{claim}
As a consequence we define:
\[ O(F_{\alpha}):= \underset{\sigma\cap\Omega_{\alpha}\neq 0}\sqcup O(\sigma). \]
We call these sets
$O(F_{\alpha})$ the points added with respect to $F_{\alpha}.$

Note that the converse to the above claim does not hold. It is thus reasonable to ask
about the other $O(\sigma)$ which are not part of $O(F_\alpha)$?
The answer is that these relate to $O(F_{\beta})$ when $F_{\alpha}$ is a boundary component of $F_{\beta}$.
Indeed, having $F_{\alpha}$ on the boundary of $F_{\beta}$ implies:
\begin{itemize}
\item$ \cU_{\beta} \subset \cU_{\alpha}$,
\item $\Omega_{\beta}$ is on the rational boundary of $\Omega_{\alpha}$, and
\item $\Sigma_{\beta}=\Sigma_{\alpha} \cap \cU_{\beta}$.
\end{itemize}
We can thus construct maps as follows
\[ \xymatrix{
  U_\beta \backslash \cD \ar[r]^{U_\beta\backslash U_\alpha} \ar@{^{(}->}[d] &  U_\alpha \backslash \cD \ar@{^{(}->}[d] \\
  (U_\beta \backslash \cD)_{\Sigma_\beta} \ar[r]^{U_\beta\backslash U_\alpha} \ar@{^{(}->}[d] &  (U_\alpha \backslash \cD)_{\Sigma_\alpha} \ar@{^{(}->}[d] \\       
  (U_\beta \backslash \cD)_{\Sigma_\beta} \ar[r]^{\pi_{\alpha,\beta}}  &  (U_\alpha \backslash \cD)_{\Sigma_\alpha} \\       
  O_\beta(\sigma') \ar[r] \ar@{^{(}->}[u] & O_\alpha(\sigma) \ar@{^{(}->}[u] 
  }
\]
where $\sigma'\in \Sigma_\beta$ has image $\sigma\in \Sigma_\alpha$.

We define a projection map $\pi_{\alpha}:(U_{\alpha} \backslash \cD)_{\Sigma_{\alpha}} \rightarrow \overline{\Gamma \backslash \cD}^{\textrm{Sat}}$ using the natural projection maps $O(F_{\alpha}) \rightarrow F_{\alpha}$.
We assert that this map is holomorphic, but note that we have not defined the topology on the Satak\'e compactification. 

\subsubsection{Gluing}

Having defined the points we wish to add, we must describe how these points will all fit together.
We first take a further quotients of the space we have constructed. In order to get a reasonable space we need the following proposition.
\begin{prop}
The action of the group $\Gamma_{\alpha}/ U_{\alpha}$ on $(U_{\alpha} \backslash \cD)_{\Sigma_{\alpha}}$ is properly discontinuous .
\end{prop}
See \cite[Sec. 3.6.3 Prop. 2]{AMRT}.

\begin{thm}
The quotient 
\[ (\Gamma_{\alpha} /U_{\alpha}) \backslash (U_{\alpha} \backslash \cD)_{\Sigma_{\alpha}}\]
 has the structure of a normal analytic space. Moreover,
\[ \overline{O(F_{\alpha})} := (\Gamma_{\alpha} /U_{\alpha})\backslash O(F_{\alpha})\]
 is an analytic subspace.
\end{thm}

We have
\[ \xymatrix{
  \Gamma\backslash \cD \ar[r] \ar@{^{(}->}[d] & \Gamma\backslash \cD \ar@{^{(}->}[d] \\
  (\Gamma_{\alpha} /U_{\alpha}) \backslash (U_{\alpha} \backslash \cD)_{\Sigma_{\alpha}} \ar@{->}[r]^{\pi_\alpha} & \overline{(\Gamma\backslash \cD)}^{\textrm{Sat}} \\
  (\Gamma_{\alpha} /U_{\alpha}) \backslash O(F_\alpha) \ar@{^{(}->}[u] \ar[r]^{\pi_\alpha} & \Gamma_\alpha\backslash F_\alpha \ar@{^{(}->}[u]
}
\]
We want the spaces $(\Gamma_{\alpha} /U_{\alpha}) \backslash (U_{\alpha} \backslash \mathcal{D})_{\Sigma_{\alpha}}$ to give us an ``open covering" of $\overline{ (\Gamma \backslash \mathcal{D})}_{\Sigma_{\alpha}}^{\textrm{tor}}$ in the sense that the maps from them give an open covering.

We give two ways to think about it.
\begin{itemize}
\item 
Firstly we consider the collection of $(\Gamma_{\alpha} /U_{\alpha}) \backslash (U_{\alpha} \backslash \mathcal{D})_{\Sigma_{\alpha}}$ modulo $\Gamma$, that is, taking one representative for each cusp of $\overline{X}^{\textrm{Sat}}$.
This is a finite collection.

Now if $\overline{F}_{\alpha}\cap F_{\beta} \supset F_{\omega}$, then we glue along the image of $\pi_{\beta,\omega}, \pi_{\alpha,\omega}$ of $(U_{\omega}\backslash \mathcal{D})_{\Sigma_{\omega}}$ in each factor.
The difficulty is that there is no map 
\[ (\Gamma_{\omega}/U_\omega) \backslash (U_{\omega} \backslash \mathcal{D})_{\Sigma_{\alpha}}\rightarrow (\Gamma_{\alpha} /U_{\alpha}) \backslash (U_{\alpha} \backslash \cD)_{\Sigma_{\alpha}}.\]
However, there is a neighbourhood of $\cB_\alpha$ on which $(U_{w}\backslash \mathcal{D})_{\Sigma_{w}}$ injects so that the map descends.

\item
Alternatively we construct the space as
\[ \indnota{\overline{ (\Gamma \backslash \mathcal{D})}_{\Sigma_{\alpha}}^{\textrm{tor}}} =
 \underset{F_{\alpha}}{\coprod} (U_{\alpha}\backslash\mathcal{D})_{\Sigma_{\alpha}}/\sim 
\]
 where we define the equivalence relation as follows. For $x_{\alpha} \in (U_{\alpha}\backslash \mathcal{D})_{\Sigma_{\alpha}}$ and $ x_{\beta} \in (U_{\alpha}\backslash \mathcal{D})_{\Sigma_{\beta}}$ we say $x_{\alpha} \sim x_{\beta}$ if there exists a boundary component $F_w$ an element $\gamma \in \Gamma$ and a point $x_w\in (U_{\omega}\backslash \mathcal{D})_{\Sigma_{\omega}}$ such that:
\begin{align*}
\pi_{\alpha,\omega}(x_\omega)&=x_{\alpha} \text{ and}\\
\pi_{\alpha,\omega}(x_\omega)&=\gamma x_{\beta}.
\end{align*}
\end{itemize}

Using either interpretation we can define the map: 
\[
\overline{\pi}_{F_{\alpha}}  :(\Gamma_{\alpha}/ U_{\alpha})\backslash(U_{\alpha}\backslash \cD)_{\Sigma_{\alpha}} \rightarrow \overline{\Gamma \backslash \cD}^{\textrm{tor}}.
\]
We note that $\overline{\pi}_{F_{\alpha}}$ is injective near $\overline{O(F_{\alpha})}$.  Consequently in a neighbourhood of $ (\Gamma_{\alpha} /U_{\alpha}) \backslash O(F_{\alpha})$ the space
$\overline{(\Gamma \backslash \mathcal{D})}^{\textrm{tor}}_{\Sigma}$
 looks like:
\[ \overline{(p_{\alpha,h}(\Gamma)\backslash F_{\alpha})}^{\textrm{tor}} \times V/{``V_{\mathbb{Z}}+\tau V_{\mathbb{Z}}"} \times (T_{\alpha})_{\Sigma_{\alpha}}. 
\] 

\subsection{Properties of Toroidal Compactifications}

We now discuss some of the properties of toroidal compactifications and how they relate to the choice of $\Sigma$. The following results are more or less clear from the construction. Details can be found in \cite{AMRT}.
\begin{enumerate}

\item The boundary has codimension $1$, ($O(\sigma_u)$ for $\sigma_u$ minimal).

\item There is a map $\overline{(\Gamma\backslash \mathcal{D})}_{\Sigma}^{\textrm{tor}} \rightarrow
\overline{( \Gamma\backslash\mathcal{D})}^{\textrm{Sat}}.$

\item The space $\overline{(\Gamma\backslash \mathcal{D})}_{\Sigma}^{\textrm{tor}}$ is not a unique, but it is functorial in $\Sigma$, and is compatible with the level structure.

\item The space is compact. 

\end{enumerate}

\subsubsection{Smoothness}

\begin{df}
A subgroup $\Gamma \in \GL_n$ is \inddef{neat} if for all $G \subset  \GL_n,$ and algebraic maps $\Phi:G\rightarrow H$ the group $\Phi(\Gamma\cap G)$ is torsion-free.
\end{df}

A key property of neat subgroups is that they act without fixed points. It is a theorem of Borel \cite[Prop. 17.4]{BorelArithmetic} that neat subgroups exist.

\begin{claim}
The singularities of $\overline{\Gamma\backslash \mathcal{D}}^{\textrm{tor}}$ are all either:
\begin{enumerate}

\item finite quotient singularities from non-neat elements of $\Gamma$, or

\item toroidal singularities in $\pi_{F_{\alpha}}(O(\sigma))$ for irregular cones $\sigma_{u}$.

\end{enumerate}
\end{claim}
This follows by observing which types of singularities can exist in the quotients of $F_\alpha \times \cV_\alpha \times (\cU_{\alpha,\bC})_{\Sigma_\alpha}$.

\begin{claim}
There exist regular $\Gamma$-admissible refinements.
\end{claim}
See \cite[p. 173]{FaltingsChai} or \cite[Sec. 4]{LooijengaL2}.

\subsubsection{Projectivity of $\overline{(\Gamma \backslash \mathcal{D})}^{\textrm{tor}}$}

For more details on this see \cite[Sec. 4.2]{AMRT}.

\begin{df}
A $\Gamma$-admissible decomposition $\Sigma_{\alpha}$ is called \inddefs{projective}{cone decomposition} if there exists functions $\varphi_\alpha:\Omega_\alpha \rightarrow \mathbb{R}^+$ which are:
\begin{enumerate}

\item convex, piecewise linear, and $\Gamma$-invariant functions for which $\varphi_\alpha(\Gamma_\alpha \cap \Omega_\alpha) \subset \bZ$, and

\item for all $\sigma \in \Sigma_\alpha$ there exists a linear functional $\ell_\sigma$ on $\mathcal{U}_\alpha$ such that:
\begin{enumerate}
\item $\ell_\sigma \ge \varphi_\alpha$ on $\Omega_\alpha$, and
\item $\sigma = \{ x\in \mathcal{U}_\alpha \mid \ell_\sigma(x) = \varphi_\alpha(x) \}$
\end{enumerate}
(equivalently $\sigma$ the maximal subsets on which $\varphi_\alpha$ is linear).
\end{enumerate}

\end{df}

Define $\varphi_\alpha^*(\lambda)=\underset{\sigma,i}\min( \vec{v}_i(\lambda))$ where $\vec{v}_i$ are vertices of $\sigma \cap \{ \varphi=1 \}$.

\begin{prop}
Every holomorphic function on $\overline{\left( \Gamma \backslash \cD\right)}^{\textrm{Sat}}$ has a Fourier expansion of the form:
\[
\sum_{\rho \in \Omega_\alpha \cap \cU_\alpha^\ast } \theta_\rho(\tau_1,\tau_2)e^{2\pi i\rho(\tau_3)}.\
\]
where $(\tau_1,\tau_2,\tau_3) \in F_\alpha\times \cV_\alpha\times \cU_{\alpha,\bC}$.
\end{prop}
See \cite[Sec. 3]{BailyFourier} for details.
This is just the development of a Fourier series with respect to $U_\alpha$.
The positivity condition on Fourier coefficients is equivalent to the growth conditions.

\begin{df}
We define a sheaf $J_m$ on $\overline{X}^{\textrm{Sat}}$ by defining the stalks to be:
\[
J_{m,x}=\{ f \in \mathcal{O}_x \mid \theta_\rho \neq 0 \textrm{ only if } \rho \in U_{\alpha}\cap \Omega, \varphi_{\alpha}^*(\rho)\geq m \}.
\]
We define the locally free sheaf $I_{\alpha}$ on $\left(U_{\alpha} \backslash \mathcal{D}\right)_{\Sigma_{\alpha}}$ to be the one generate by $e^{2\pi i\varphi_\alpha(\tau_2)}$.
We then define the sheaf $I$ on $\overline{(\Gamma \backslash \cD)}^{\textrm{tor}}$ by:
\[ \Gamma(U,I)=\{ s \in \underset{\alpha}\oplus \Gamma(\pi_\alpha^{-1}(U),I_\alpha)  \mid \textrm{`glue on overlaps'}\}.\]
\end{df}
We have that:
\[ J_m = \pi_{\alpha\ast}( I^m). \]

\begin{thm}
The toroidal compactification $\overline{(\mathcal{D} / \Gamma)}^{\textrm{tor}}_\Sigma$ is the normalization of the blow up of $\overline{(\mathcal{D} / \Gamma)}^{\textrm{Sat}}$ along $J_m$.
Moreover, $\pi^{*}_\alpha (J_m)=I^m.$ 
\end{thm}
See \cite[Sec. 4.2.1]{AMRT}.

\subsection{Toroidal Compactification for the Orthogonal group}

We now summarize all the objects we shall need for toroidal compactifications in the case of orthogonal groups.

      We will assume that:
      \[ \tilde{A} = \begin{pmatrix} 0 & 0 & 1 & 0 & \\
                         0 & 0 & 0 & 1 & \\ 
                         1 & 0 & 0 & 0 & \\ 
                         0 & 1 & 0 & 0 & \\ 
                           &   &   &   & A \end{pmatrix} \]
      gives the matrix for our quadratic space over $\bQ$.
      Note that for $n\leq4$ there may be no such matrix over $\bQ$. The main difference in the theory if no such matrix of this form exists is that certain classes of boundary components will simply not exist. Though these cases are of interest, we will not treat them here.

\subsubsection{Boundary Components and Parabolics}

We now compute the shape of the parabolics for the different boundary components.

These parabolics $P_\alpha = \{g\in G \mid gF_\alpha=F_\alpha\} $ come from fixing  a real isotropic subspace $\alpha$. The group $P_\alpha$ is the stabilizer in $G$ of this space. 
      Up to equivalence the options for $\alpha$ are $\{ e_1 \},\{ e_1,e_2 \} $. It is conceivable that there may be no rational parabolics of one or both types. This can only happen if $n$ is small, and cannot happen based on our assumption about the shape of the quadratic form.
We have that the corresponding $P_\alpha$ have the following form:
       \begin{itemize}
      \item $\{ e_1 \}$
      \[ 
\begin{pmatrix} a & x_1 & x_2  & x_3  & x_4  \\
                                    0 & \ast^{(1)}_3 &  y_3 & \ast_3^{(3)} & \ast_3^{(4)} \\
                                    0 & 0 &a^{-1}&    0 & 0    \\
                                    0 & \ast^{(1)}_1 & y_1  & \ast_1^{(3)} & \ast_1^{(4)} \\
                                    0 & \vec{\ast}^{(1) t}_4 & \vec{y_4}& \vec{\ast}_4^{(3) t} & \ast \end{pmatrix}  
\hspace{-2.3cm}
\begin{split} 
&a\in \bG_m \\
&(\ast_{ij}) \in\SO_{1,n-1}\\
&x_2= -a(y_3y_1 +\tfrac{1}{2} \vec{y_4}A\vec{y_4}^t)\\
&x_i= -a(\ast^{(i)}_3y_1+\ast^{(i)}_1y_3+ \vec{\ast^{(i)}_4}A\vec{y_4}^t)\; i\neq 2.\\
\end{split}\]
                   
      \item $\{ e_1,e_2 \}$
      \[ 
\begin{pmatrix} a & b  & x_2  & x_3  & \vec{x_4}  \\
                        c & d  & w_1  & w_2  & \vec{w_4}  \\
                        0 & 0  & d'   & -b'  & 0        \\
                        0 & 0  & -c'  & a'   & 0        \\
                        0 & 0  &\vec{y}^t& \vec{z}^t& \ast \end{pmatrix}
\hspace{-2cm}
\begin{split}
     &\smtx abcd \in \GL_2, (\ast_{ij}) \in \SO_A\\
     &\smtx{d'}{-c'}{-b'}{a'} = (\smtx{a}{b}{c}{d}^{-1})^t \\
     &d'x_2-c'w_1=-\tfrac{1}{2}\vec{y}A\vec{y}^t \\
     &-b'x_3+a'w_2=-\tfrac{1}{2}\vec{z}A\vec{z}^t \\ 
     &d'x_3-c'w_2-b'x_2+a'w_1=-\vec{y}A\vec{z}^t\\
     &d'x_i-c'w_i=-\tfrac{1}{2}\vec{\ast}^{(i)}A\vec{y}^t \text{ for } i \ge 4\\
     &-b'x_i+a'w_i=-\tfrac{1}{2}\vec{\ast}^{(i)}A\vec{z}^t \text{ for } i \ge 4.\\ 
\end{split}     \]
      \end{itemize}
The unipotent radical $\cW_\alpha$ of $P_\alpha$ is:
       \begin{itemize}
      \item $\{ e_1 \}$
      \[ 
\begin{pmatrix} 1 & x_1 & x_2  & x_3  & \vec{x_4} \\
                        0 & 1   & y_3  & 0    & 0    \\
                        0 & 0   &1     & 0    & 0    \\
                        0 & 0   & y_1  & 1    & 0    \\
                        0 & 0   &\vec{y_4}^t&  0 & \Id \end{pmatrix}
\hspace{-2cm}
\begin{split}
       &x_2= -(y_1y_3 + \tfrac{1}{2}\vec{y_4}A\vec{y_4}^t)\\
       &x_3=-y_3\\
       &x_1=-y_1\\    
       &x_i = -y_iA_{i-3i-3} \text{ for } i\ge4.\end{split}\]

      \item $\{ e_1,e_2 \}$
      \[ \begin{pmatrix} 1 & 0  & x_2  & x_3  & \vec{x_4}  \\
                                    0 & 1  & w_1  & w_2  & \vec{w_4}  \\
                                    0 & 0  & 1    & 0    & 0        \\
                                    0 & 0  & 0    & 1    & 0        \\
                                    0 & 0  &\vec{y_4}^t& \vec{z_4}^t& \Id \end{pmatrix} 
\hspace{-2cm}
\begin{split}
      &x_2 = -\tfrac{1}{2}\vec{y}A\vec{y}^t\\      
      &w_2 = -\tfrac{1}{2}\vec{z}A\vec{z}^t\\     
      &x_3+w_1=-\vec{y}A\vec{z}^t\\
      &x_i=-y_iA_{i-3i-3} \text{ for } i\ge4\\      
     &w_i=-z_iA_{i-3i-3} \text{ for } i\ge4. \end{split} \]

      \end{itemize}
We therefore find the centre $\cU_\alpha$ of $\cW_\alpha$ is:
      \begin{itemize}
      \item $\{ e_1 \}$
       \[  \begin{pmatrix} 1 & x_1 & x_2  & x_3  & \vec{x_4} \\
                        0 & 1   & y_3  & 0    & 0    \\
                        0 & 0   &1     & 0    & 0    \\
                        0 & 0   & y_1  & 1    & 0    \\
                        0 & 0   &\vec{y_4}^t&  0 & \Id \end{pmatrix}
\hspace{-2cm}
\begin{split}
       &x_2= -(y_1y_3 +  \tfrac{1}{2}\vec{y_4}A\vec{y_4}^t)\\
       &x_3=-y_3\\     
       &x_1=-y_1\\     
       &x_i = -y_iA_{i-3i-3} \text{ for } i\ge4. \end{split} \]
When we need to denote this compactly, we write $\cU_\alpha = \{ (y_1,y_3,\vec{y_4}) \}$.

      \item $\{ e_1,e_2 \}$
      \[   \begin{pmatrix}1 & 0  & 0    & x_3  & 0    \\
                        0 & 1  & w_1  & 0    & 0   \\
                        0 & 0  & 1    & 0    & 0        \\
                        0 & 0  & 0    & 1    & 0        \\
                        0 & 0  & 0    & 0    & \Id \end{pmatrix}
\hspace{-2cm}
\begin{split}
      x_3 = - w_1.
\end{split} \]
When we need to denote this compactly, we write $\cU_\alpha = \{ (w_1) \}$.
      \end{itemize}
We can now describe $\cV_\alpha = \cW_\alpha/\cU_\alpha$.
      \begin{itemize}
       \item $\{ e_1 \}$
       
      This is trivial.
       
       \item $\{ e_1,e_2 \}$

       Coset representatives are given by $(\vec{y_4},\vec{z_4})$. The identification with $\vec{v_4} = \vec{y_4} + i\vec{z_4}$ gives it a complex structure.
       \end{itemize}

We now describe the realization of $\cD$ as a Siegel domain of the third kind. We have $\cB_\alpha = \cU_{\alpha,\bC}\cdot \cD \subset P_+$ and the diagram:
\[ \xymatrix{ \cB_\alpha = \cU_{\alpha,\bC}\cdot D  \ar@{->}[d] 
                                        & \simeq & F_\alpha \times \cV_\alpha\times \cU_{\alpha,\bC}\ar@{->}[d] \\
              \cB_\alpha/\cU_{\alpha,\bC} \ar@{->}[d]& \simeq & F_\alpha\times \cV_\alpha \ar@{->}[d] \\
              \cB_\alpha/(\cW_\alpha\cdot\cU_{\alpha,\bC}) & \simeq & F_\alpha.
              }
              \]

      \begin{itemize}
      \item $\{ e_1 \}$

           As $\cV_\alpha,F_\alpha$ are trivial the identification $\cB_\alpha = \cU_\alpha$ is apparent.
           We describe it in $P(V(\bC))$. We observe that:
           \begin{align*} \cU_{\alpha,\bC}\cdot \cD &= \cU_{\alpha,\bC}\cdot [1:i:1:i: \vec{0}]\\
                                   & =  [1-iy_1-iy_3-(y_1y_3 + \tfrac{1}{2} \vec{y_4}^tA\vec{y_4}): i+y_3: 1 : i+y_1 : \vec{y_4} ]. 
           \end{align*}
           It is apparent from this that the map from $P(V(\bC))$ to $\cU_{\alpha,\bC}$ is given by:
 \[ 
     [(v_0,\ldots,v_{n+1})] \mapsto     (y_1,y_3, \ldots,y_n) =  ( \tfrac{v_1}{v_2},\tfrac{v_3}{v_2}\ldots,\tfrac{v_{n+1}}{v_2}  ).
\]
Note the use of $y_1,y_3$ rather than $y_1+i,y_3+i$.

       \item $\{e_1, e_2 \}$

       It is more convenient to express the action inside the orbit above so $\cU_{\alpha,\bC}\cdot \cD$ is:
          \begin{align*}\cU_{\alpha,\bC}&\cdot [1-iy_1-iy_3-(y_1y_3 + \tfrac{1}{2} \vec{y_4}^tA\vec{y_4}): i+y_3: 1 : i+y_1 : \vec{y_4} ]   \\
          &= [1-iy_1-i(y_3+w_1)-(y_1(y_3+w_1) + \tfrac{1}{2} \vec{y_4}^tA\vec{y_4}): \\&\hspace{5cm} i+(y_3+w_1): 1 : i+y_1 : \vec{y_4} ].
    \end{align*}
We compose this with the inverse above and conclude we have the following:
\[ \cV_\alpha \simeq \{ [\tfrac{1}{2} \vec{v_4}^tA\vec{v_4} : 0 : 1 : 0 : \vec{v_4} ] \in \kappa^+ \} \] 
in the sense that $\cV_\alpha = (\vec{y_4},\vec{z_4}) \mapsto \vec{y_4}+i\vec{z_4} = \vec{v_4}$.
The map $\cB_\alpha \rightarrow \cU_{\alpha,\bC}$ given by:
\[  [-(v_1v_3 + \tfrac{1}{2} \vec{v_4}^tA\vec{v_4}): v_3: 1 :v_1 : \vec{v_4} ] 
 \mapsto v_3.
\]
Finally, we have the map $\cB_\alpha \rightarrow F_\alpha \simeq \uhp$ given by:
\[  [-(v_1v_3 + \tfrac{1}{2} \vec{v_4}^tA\vec{v_4}): v_3: 1 :v_1 : \vec{v_4} ] 
 \mapsto v_1.
\]

      \end{itemize}

We now describe the self-adjoint open cone $\Omega_\alpha \subset \cU_\alpha$.
      \begin{itemize}
       \item $\{ e_1 \}$
       \[ \Omega_\alpha = \{ (y_1,\ldots,y_n) \in \cU_\alpha \mid y_1y_3 +  \tfrac{1}{2}\vec{y_4}A\vec{y_4} > 0 \text{ and } y_3 > 0 \}. \]
       \item $\{ e_1,e_2 \}$
        \[ \Omega_\alpha = \{ (w_1) \in \cU_\alpha \mid w_1>0 \}. \]
       \end{itemize}
It comes with a map:
\[ \Phi_\alpha : \cB_\alpha = P_\alpha\cdot \cU_{\alpha,\bC}/(P_\alpha\cap K) \rightarrow P_\alpha\cdot \cU_{\alpha,\bC}/P_\alpha \simeq \cU_\alpha. \]
         \begin{itemize}
          \item $\{ e_1 \}$
          \[ \Phi_\alpha: \vec{y} \mapsto (\Im(y_1), \Im(y_3),\Im(\vec{y_4})). \]
         \item $\{e_1,e_2 \}$
          \[ \Phi_\alpha: \vec{y} \mapsto (2\Im(y_1)\Im(y_3) + \Im(\vec{y_4})^tA\Im(\vec{y_4})). \]
         \end{itemize}
We may check that $\cD = \Phi_\alpha^{-1}( \Omega_\alpha)$ in either case.

We now look at the Levi decomposition for $P_\alpha$.
   We have the subgroups $G_{h,\alpha},G_{\ell,\alpha},m_\alpha \subset P_\alpha$. These satisfy
 $P_\alpha \simeq (G_{h,\alpha}\cdot G_{\ell,\alpha}\cdot m_\alpha)W_\alpha$ with
      $m_\alpha$ being compact.
       \begin{itemize}
       \item $\{ e_1 \}$
\subitem
       $G_{h,\alpha}$ is trivial.
\subitem
       $G_{\ell,\alpha}$ is $\bG_m\times \SO_{1,n-1}$.
\subitem
       $m$ is trivial.
       \item $\{ e_1,e_2 \}$
\subitem
        $G_{h,\alpha} = \SL_2$.
\subitem
        $G_{\ell,\alpha} = \bG_m$ viewed as the diagonal  in the apparent $\GL_2$ factor.
\subitem
        $m = \SO_{n-2}$.
       \end{itemize}
This decomposition is characterized by two maps. The map $p_{h,\alpha} : P_\alpha\surjects G_{h,\alpha}\sim\Aut(F_\alpha)$.
       \begin{itemize}       
       \item $\{ e_1 \}$

             This is the trivial map.
       \item $\{ e_1,e_2 \}$

             This is the map $g\mapsto \overline{\smtx abcd}$ where we view this in $\PGl_2 = \Aut(\bH)$ under the action 
\[ \smtx abcd \circ [1:v_1] = [1:\tfrac{av_1-b}{-cv_1+d}].\]
             We see immediately that this map is equivariant for the action of $P_\alpha$.
       \end{itemize}
The map $p_{\ell,\alpha} : P_\alpha\surjects G_{\ell,\alpha}=\Aut(\cU_\alpha,\Omega_\alpha)$
      (the group $ G_{\ell,\alpha}$ acts on $\cU_\alpha$ by conjugation).
       \begin{itemize}
          \item $\{ e_1 \}$

             This is the map $\bG_m\times \SO_{1,n-1} \rightarrow \GO_{1,n-1}$.
             (It is the connected component of the identity which preserves the cone.)
          \item $\{ e_1,e_2 \}$

             This is the map $g\mapsto \det{\smtx abcd}$ where we view $\det{\smtx abcd}\in\bG_m$.        
       \end{itemize}
In both cases we can check that the maps are equivariant.

We also have the following objects:
\begin{itemize}
\item $\Gamma_\alpha=\Gamma\cap P_\alpha$.
\item $\Gamma_\alpha'= \Gamma_\alpha \cap \Ker(p_{\ell,\alpha})$.
\item $\overline{\Gamma}_\alpha = p_{\ell,\alpha}(\Gamma_\alpha) \subset \Aut(U_\alpha,\Omega_\alpha)$.

\item $1\rightarrow \Gamma'_\alpha\rightarrow\Gamma_\alpha\rightarrow\overline{\Gamma}_\alpha\rightarrow 1$.
\item $U_\alpha=\Gamma\cap \cU_\alpha$ a lattice.
\item $W_\alpha=\Gamma\cap \cW_\alpha$.
\item $W_\alpha/U_\alpha\subset \cV_\alpha$ a lattice.
\end{itemize}

\subsubsection{Partial Quotient and Boundary Components}

The open neighbourhoods of the cusps that we need to consider are the spaces $U_\alpha\backslash \cB_\alpha$. These are:
\begin{itemize}
 \item $\{ e_1 \}$ 
\[
            U_\alpha\backslash \cB_\alpha = U_\alpha \backslash \cU_{\alpha,\bC} \simeq (\bC^\times)^n.
\]
We shall add points ``near the origin" of $\bC^\times$. These points will correspond to the infinite limit points of $\Omega_\alpha$.

\item $\{ e_1,e_2 \}$
\[
            U_\alpha\backslash \cB_\alpha =   F_\alpha \times V_\alpha \times (U_\alpha\backslash \cU_{\alpha,\bC}) = \bH \times \bC^{n-2}\times \bC^\times.
\]
We shall add the point ``at the origin" of $\bC^\times$ which corresponds to the point at infinity of $\bR^+ = \Omega_\alpha$.
             
\end{itemize}

We now consider the further quotient modulo $\overline{\Gamma}_\alpha$.
\begin{itemize}
 \item $\{ e_1 \}$ 

The group is $\overline{\Gamma}_\alpha \subset SO_{1,n-1}(\bZ)$.
            We thus wish to consider $\overline{\Gamma}_\alpha \backslash (\bC^\times)^n_{\Sigma_\alpha}$.               
            The action on the interior of $\cD$ is non-trivial. However, on the $O(\sigma)$ components of the cusps it simply acts to identify them so that locally near the cusps, everything looks the same.

\item $\{ e_1,e_2 \}$ 

The group $\overline{\Gamma}_\alpha \subset \{\pm 1\}$ acts trivially.
         
\end{itemize}

We next consider the quotient modulo $\Gamma'_\alpha$.
\begin{itemize}
\item $\{ e_1 \}$ 

 These groups are trivial, hence there is no action.

\item $\{ e_1,e_2 \}$

The group is $\Gamma'_\alpha \subset \SL_2(\bZ) \ltimes (\bZ^2)^{n-2}$.
               This acts trivially on the $\cU_{\alpha,\bC}$ component of $\cB_\alpha$. we thus consider its action on $F_\alpha \times \cV_\alpha = \bH \times \bC^{n-2}$.
We see that the matrix:
\[ \begin{pmatrix} a & b  & \ast  & \ast  & \vec{\ast}  \\
                        c & d  & \ast  & \ast  & \vec{\ast}  \\
                        0 & 0  & d'   & -b'  & 0        \\
                        0 & 0  & -c'  & a'   & 0        \\
                        0 & 0  &\vec{y}& \vec{z}& X \end{pmatrix} \in \Gamma'_\alpha \]
sends:
\[ [\ast:Z:1:v_1:\vec{v_4}] \in \cB_\alpha]  \mapsto [ \ast: Z' : 1 : \tfrac{av_1-b}{-cv_1+d} : \tfrac{1}{-cv_1+d}(\vec{y}_4 + v_1\vec{z}_4 + X\vec{v}_4) ]. \]
Due to the equivariance of the action on $\cU_\alpha$ we have $Z'=0$ if and only if $Z=0$.
Thus, in a small neighbourhood $N$ of $X=0$ in $\cU_\alpha$ we have that $\Gamma'_\alpha \backslash \cB_\alpha$ is of the form:
\[ N \times (\Gamma'_\alpha \backslash( \bH \times \bC^{n-2})). \]
Define $\cE^{(n-2)}$ to be the quotient:
\[ \cE^{(n-2)} :=\Gamma_\alpha' \backslash (\bH \times \bC^{n-2}) \] for the action described above. It comes with a map:
\[ \cE^{(n-2)} \surjects p_{h,\alpha}(\Gamma_\alpha')\backslash \uhp \]
where the right hand side is the modular curve $Y(p_{h,\alpha}(\Gamma_\alpha'))$.
The fibres of the morphism $\cE^{(n-2)} \rightarrow Y(p_{h,\alpha}(\Gamma_\alpha'))$ satisfy:
\[ \cE^{(n-2)}_\tau = (\cE_\tau)^{n-2} \]
where $\cE_\tau = E_\tau$ is the elliptic curve with level structure parametrized by $\tau\in  Y(p_{h,\alpha}(\Gamma_\alpha'))$.
Thus we see that $\cE^{(n-2)} = \cE \times_Y \cdots \times_Y \cE$ is the $(n-2)$-fold fibre product of the universal elliptic curve over the modular curve.
\end{itemize}

\subsubsection{Adjacent boundary components}

We now describe the relations between adjacent boundary components.
As there are two types of boundary components, there are naturally two types of adjacency to consider.

We shall first consider the case where $F_\alpha$ is $1$-dimensional and $F_\beta$ is $0$-dimensional. We say these are adjacent if $F_\beta \subset \overline{F_\alpha}$ in $\overline{X}^{\textrm{Sat}}$.

In this case $F_\alpha$ corresponds to a $2$-dimensional isotropic subspace containing the $1$-dimensional isotropic subspace corresponding to $F_\beta$.
It follows that the parabolics $P_\alpha$ and $P_\beta$ are simultaneously conjugate to our standard ones above.
We see that $\cU_\alpha \injects \cU_\beta$ and admissibility of the cone decomposition implies that the image of $\Omega_\alpha$ (which is $1$-dimensional) is a cone in $\overline{\Omega}^{\textrm{rat}}_\beta$.

We may also view $F_\beta$ as a boundary component of $F_\alpha \simeq \uhp$. Thus $F_\beta$ corresponds to a cusp of $Y(p_{h,\alpha}(\Gamma_\alpha'))$. Without loss of generality it is the cusp $i\infty$.
We see that as $\tau\in \uhp$ approaches $F_\beta$, the lattice we are taking a quotient by to get $\cE^{(n-2)}$ is degenerating to:
\[ \bZ^{n-2} \times i\infty \bZ^{n-2} = \bZ^{n-2} \subset \cU_\alpha. \]
We then see quite naturally that we have a map:
\[  F_\alpha \times ((U_\beta\cap V_\alpha)\backslash \cV_\alpha) \times U_\alpha\backslash \cU_{\alpha,\bC} = F_\alpha \times (\bC^\times)^{n-1} \rightarrow (\bC^\times)^n=  U_\beta \backslash \cU_\beta \]
using the map $\uhp \overset{\textrm{exp}}\rightarrow \bC^\times$.

Looking at the cone $\Omega_\alpha \in \cU_\alpha$ we see that:
\[  (U_\alpha\backslash F_\alpha \times V_\alpha \times \cU_{\alpha,\bC})_{\Omega_\alpha} \subset U_\beta\backslash \cU_\beta \sqcup O(\Omega_\alpha) \]
where we are viewing $O(\Omega_\alpha)$ relative to $\cB_\beta$. We have further that:
\[ O(\Omega_\alpha) \simeq ((U_\alpha \cap\Gamma_\beta)\backslash \uhp) \times (V_\beta \cap U_\alpha) \backslash \cV_\beta \times (0) \rightarrow \cE^{(n-2)}. \]
We now wish to describe the closure of $\indnotalpha{\cE^{(n-2)}}{Ec}$ in the compactification.

Given any other cone $\sigma \in \Sigma_\beta$ we see that $O(\sigma)$ is in the closure of $O(\Omega_\alpha)$ in $\cB_\beta$ if and only if $\Omega_\alpha \subset \sigma$.
We thus consider the set:
\[ (\Sigma_\beta)_\alpha := \{ \sigma \in \Sigma_\beta \mid \Omega_\alpha \subset \sigma \}. \]
We now consider the image of $(\Sigma_{\beta})_{\alpha}$ in $\cU_\beta / \cU_\alpha \simeq \bR^+ \times \cV_\alpha$.
This gives us a cone decomposition for the cone associated to $\uhp \times \cV_\beta$.
This cone decomposition is invariant under the action of $\Gamma_\alpha\cap \Gamma_\beta$. 
Indeed $(U_\beta\cap V_\alpha)\backslash V_\beta \injects G_{\ell,\alpha}$ must stabilize the collection of cones adjacent to $\Omega_\alpha$ and since $\cV_\beta$ centralizes $\cU_\beta$ this action descends to $\cU_\alpha / \cU_\beta$. We likewise find that the stabilizer of $i\infty$ in $G_{h,\beta}$ injects into $\cW_\alpha$, and thus also stabilizes $\Sigma_{\beta,\alpha}$. Moreover, as $V_\alpha = U_\beta\cap \cV_\alpha$ it is also rational for the appropriate rational structure.
Consequently, if we proceed as in the usual construction of toroidal compactification we can construct:
\[
\overline{\cE^{(n-2)}}_{\Sigma_{\beta,\alpha}} :=
((\overline{\Gamma}_\beta \ltimes V_\alpha) \backslash (\uhp \times (\bC^\times)^{n-2})_{(\Sigma_\beta)_\alpha}).
\]
 This map is injective near the cusp $F_\beta$ and realizes a compactification of $\cE^{(n+2)}$ near this point.
Moreover, by functoriality we obtain a map:
\[ \overline{\cE^{(n-2)}}^{\textrm{tor}}_{\Sigma_{\beta,\alpha}} \rightarrow \overline{X}^{\textrm{tor}}_{\Sigma} \]
which lands in the fibre over $F_\beta$.

We now consider the case where both $F_\alpha$ and $F_\beta$ are both $1$-dimensional.
Both boundary components are characterized by $2$-dimensional isotropic subspaces. The curves $\overline{F_\alpha}$ and $\overline{F_\beta}$ intersect in $\overline{X}^{\textrm{Sat}}$ if and only if the associated $2$-dimensional spaces intersect in a $1$-dimensional isotropic space.
In this case, there exists the boundary component $F_\omega$ corresponding to this $1$-dimensional isotropic space and $F_\omega = \overline{F}_\alpha \cap \overline{F}_\beta$. It follows that $F_\omega$ is adjacent to both $F_\alpha$ and $F_\beta$ in the sense described above.
However, there is in general no reason for the closure of the fibres over $F_\alpha$ and $F_\beta$ to intersect in $\overline{X}^{\textrm{tor}}_\Sigma$ just because the images intersect in $\overline{X}^{\textrm{Sat}}$.
From the discussion in the previous case, it is apparent that the closures of the fibres will intersect in $\overline{X}^{\textrm{tor}}_\Sigma$ if and only if the cones $\Omega_\alpha$ and $\Omega_\beta$ viewed in $\Omega_\omega$ are both contained in a common cone $\sigma \in \Sigma_\omega$. In this case, the intersection of the closure of the fibres is precisely:
\[ \underset{\sigma \supset \Omega_\beta,\Omega_\alpha}\cup O(\sigma). \]
\begin{rmk}
We remark that if $\Sigma$ is regular then this intersection (provided it is non-empty) has dimension $n-2$.
\end{rmk}

\begin{rmk}
Even though $\Orth_V(\bQ)$ will act transitively on the set of boundary components it is not in general true that there exists a lattice $L\subset V$ such that $\Orth_L(\bZ)$ will act transitively on either the $0$-dimensional or $1$-dimensional boundary components.

However, if the Hasse invariant is trivial, then there exists a lattice $L\subset V$ with square free discriminant.
For such a lattice, the primitive representative for every isotropic vector can be embedded into a hyperplane $H$ which is a direct factor of $L$.
If $\Orth_V$ has $\bQ$ rank $2$, it follows that the isomorphism class of $H^\perp$ is uniquely determined and thus $\Orth_L(\bZ)$ will act transitively on the $0$-dimensional boundary components.

This argument fails for $1$-dimensional boundary components as: 
\[ H\oplus H \oplus E_8\oplus E_8 \simeq H\oplus H \oplus  D_{16}^+. \]
\end{rmk}

\subsection{Constructing Rational Polyhedral Cone Decompositions}

We now introduce a method for the construction projective rational polyhedral cone decomposition. This is largely a summary of the method outlined in\cite[Section 2.5]{AMRT}, See also \cite{LooijengaL2}.
Before proceeding we should note that the resulting cone decompositions need not be regular. 

We first introduce the notation we shall be using throughout.
\begin{itemize}
\item $L$ a lattice with a positive-definite bilinear form $\langle\cdot,\cdot\rangle$.
\item $L^\#$ the dual of $L$ with respect to $\langle\cdot,\cdot\rangle$.
\item $\Omega$ is a convex open homogeneous cone in $V=L\otimes\bR$ self-adjoint with respect to $\langle\cdot,\cdot\rangle$.
\item $\Gamma$ a subgroup of $\Aut_L(\Omega,V)$.
\end{itemize}

\begin{df}
A subset $K$ is said to be a \inddef{kernel} of $\Omega$ if: $0\notin \overline{K}$ and $K+\Omega\subset K$.

We say two kernels are \inddefs{comparable}{kernel} if $\lambda K'\subset K \subset \lambda^{-1} K'$.

The \inddefs{semi-dual}{kernel} of a set $A$ is:
\[ \indnota{A^\vee} = \{ h\in \Hom(V,\bR) \mid h(a) \ge 1 \text{ for all } a\in A\}.\]

The \inddef{extreme points} of a convex set $A$ are:
 \[ \indnota{E(A)} = \{ x\in \overline{A} \mid x=\frac{y+z}{2} \Rightarrow y=z=x \}. \]
\end{df}

We summarize a few key results of \cite[Sec 2.5.1-2]{AMRT}.
\begin{prop}
For a kernel $K$ we have the following:
\begin{itemize}
\item $K^\vee$ is a kernel
\item $K = \underset{e\in E(K)}\cup e+\Omega$.
\end{itemize}
\end{prop}
\begin{prop}
The closed convex hull of $\Omega\cap L$ is a kernel for $\Omega$. Moreover, these are all comparable independently of $L$.
\end{prop}

\begin{df}
A kernel is called a \inddefs{core}{kernel} if $K$ is comparable to the closed convex hull of $\Omega\cap L$. It is called a \inddefs{co-core}{kernel} if $K^\vee$ is a core.
\end{df}

\begin{ex}
We have the following examples of cores:
\begin{itemize}
\item 
$K_{\textrm{cent}}$ the closed convex hull of $\Omega\cap L$ is a core.
\item
$K_{\textrm{cent}}'$ the closed convex hull of $\Omega\cap L^\#$ is a core.
\item
$K_{\textrm{perf}} = (\text{closed convex hull of } \overline{\Omega}\cap L\setminus 0)^\vee$ is a core.
\end{itemize}
\end{ex}

\begin{df}
A closed convex kernel is called \inddefs{locally rationally polyhedral}{kernel} if for any rational polyhedral cone $\Pi$ whose vertices are in $\overline{\Omega}$ there exists a finite collection of $x_i\in V_\bQ\cap \overline{\Omega}$ such that:
\[ \Pi\cap K = \{ y\in \Pi \mid \langle x_i,y\rangle \ge 1 \}. \]
It is said to be \inddefsalpha{$\Gamma$-polyhedral}{kernel}{gamma-polyhedral} if it is moreover $\Gamma$-invariant. 
\end{df}

\begin{nota}
Let $T\subset \frac{1}{N}L\cap\overline{ \Omega}\setminus 0$ we define: 
\[ K_T = \{ x\in\overline{\Omega} \mid \langle x,y\rangle > 1 \text{ for all } y\in T \}.\]
\end{nota}

\begin{prop}
If $T$ is stable under the action of $\Aut(L^\#,\Omega)$, then $K_T$ is $\Aut(L,\Omega)$-polyhedral.
If $K$ is $\Aut(L^\#,\Omega)$-polyhedral then $K^\vee$ is $\Aut(L,\Omega)$-polyhedral.
\end{prop}
See \cite[Sec. 2.5.2 Prop. 9,10]{AMRT}.

\begin{df}
For a convex set $A\subset V$, a hyperplane $H$ is said to \emphh{support} $A$ if $A\setminus H$ is connected and $A\cap H\neq \varnothing$.

For $y\in\overline{\Omega}$ denote by $H_y := \{x\in V \mid \langle x,y \rangle = 1 \}$ the associated hyperplane.
Given a kernel $K$ define: 
 \[ \cY_K = \{y\in\overline{\Omega} \mid H_y  \text{ supports }K, H_y\cap E(K) \text{ spans } V \}.\] 

For $\underline{y} = \{y_1,\ldots, y_m\} \subset \cY_K$ let $\sigma_{\underline{y}}$ be the cone generated by $\cap_i H_{y_i} \cap E(K)$.
\end{df}

\begin{prop}
Let $K$ be a $\Gamma$-polyhedral co-core for $\Omega$ and define:
\[ \Sigma := \{ \sigma_{\underline{y}} \mid \underline{y} \subset \cY_K \text{ finite} \}. \]
The decomposition $\Sigma$ is $\Gamma$-admissible and projective.
\end{prop}
For the proof of the first statement see \cite[Sec. 2.5.2 Prop. 8]{AMRT} for the proof of projectivity see \cite[Sec. 4.2]{AMRT}.

What the above theorem does is it translates the abstract problem of finding a $\Gamma$-admissible cone decomposition into the concrete problem of understanding how the extreme points of a lattice intersect hyperplanes.
This should not be assumed to be a simple task.

\begin{prop}
Taking iterated subdivisions of a $\Gamma$-admissible and projective cone decomposition preserves $\Gamma$-admissibility and projectivity. By this process one may construct a projective regular cone decomposition.
\end{prop}
See \cite[Sec. 4]{LooijengaL2}.

}

{\def\XMetaCompile{1}

\section{Dimension Formulas for Spaces of Modular Forms}
\label{sec:SectionDimension}

One very natural question which remains unanswered about modular forms on orthogonal symmetric spaces is that of giving explicit formulas for the dimensions of spaces of modular forms on these spaces. These types of formulas have a wide variety of applications, both computational and theoretical.
This problem has been extensively studied in lower dimensional cases where exceptional isomorphisms exist between the orthogonal Shimura varieties and other classical varieties. In particular, the (2,1)-case corresponds to the classical modular and Shimura curves and the (2,2)-case corresponds to Hilbert modular surfaces. Many results are known for these cases (see for example \cite[Ch. 3]{DiamondShurman} and \cite[Ch. 2]{Freitag_HMF}). Additionally, the split (2,3)-case corresponds to a Siegel space where the work of Tsushima (see \cite{Tsushima}) gives us dimension formulas. The only work in the general case is that of \cite{GHSOrthogonal}. They are able to compute asymptotics for the dimensions as one changes the weight for several higher dimension cases.
The standard approach to this type of problem and the one we intend to discuss is that which has been used successfully in the above listed cases.

The first tool we shall discuss is the Riemann-Roch formula.

\subsection{Hirzebruch-Riemann-Roch Theorem}

Before discussing the theorem we shall quickly survey the objects involved in the statement of this theorem. Most of what we say can be found in \cite[Appendix A]{Hartshorne}. More thorough treatments exist, both from a more topological approach \cite{HirzebruchRR} or algebraic approach \cite{BorelSerreRR}.

What the Hirzebruch-Riemann-Roch theorem fundamentally is about is a formula for the Euler characteristic in terms of the values of intersection pairings between certain cycles and cocycles. We will say very little about what this means. Two good references for this material are \cite{FultonIntersection1,FultonIntersection2}.

\subsubsection{Chern Classes}

The main cohomology classes involved in the Riemann-Roch theorem are the Chern classes. There are many ways to introduce them; for an alternate topological approach see \cite{MilnorStasheff}. We mostly introduce notation and key results we shall use.

\begin{nota}
Let $\cE$ be a locally free sheaf of rank $r$ on a non-singular projective variety $X$ of dimension $n$.
Let $\mathbf{P}(\cE)$ be the associated projective space bundle (see \cite[II.7]{Hartshorne}). 
Denote by $\indnotalpha{\CH^r(X)}{chowring}$ the Chow ring of $X$, that is, the codimension $r$ cycles up to equivalence. 
Let $\xi\in \CH^1(\mathbf{P}(\cE))$ be the class of the divisor corresponding to $\cO_{\mathbf{P}(\cE)}(1)$. Let $\pi:\mathbf{P}(\cE) \rightarrow X$ be the projection.
Denote by $\indnotalpha{\cT_X}{tangeantbundle}$ the tangent sheaf of $X$ and by $\indnotalpha{\Omega_X^1}{omegaone}$ the cotangent sheaf of $X$.
\end{nota}

\begin{df}
For $i = 0,\ldots,r$ we define the $i$th \inddef{Chern class} $\indnota{c_i(\cE)}\in \CH^i(X)$ by the conditions $c_0(\cE) = 1$ and
\[ \sum_{i=0}^r (-1)^i\pi^\ast c_i(\cE)\cdot\xi^{r-i} = 0. \]
We define the \inddef{total Chern class}
\[ \indnota{c(\cE)} = c_0(\cE) + c_1(\cE) + \cdots + c_r(\cE), \]
and the \inddef{Chern polynomial}
\[ \indnota{c_t(\cE)} = c_0(\cE) + c_1(\cE)t + \cdots + c_r(\cE)t^r. \]
For a partition $\alpha = (\alpha_1,\ldots,\alpha_m)$ of $i = \sum_\ell \alpha_\ell$ we shall write $\indnota{c^{\alpha}(\cE)} = \prod_{\ell} c_{\alpha_\ell}(\cE)$.
\end{df}

\begin{prop}
The following properties uniquely characterize the Chern classes.
\begin{enumerate}
\item If $\cE = \cO_X(D)$, then $c_t(\cE) = 1+Dt$.
\item If $f:X'\rightarrow X$ is a morphism, then for each $i$ we have $c_i(f^\ast\cE) = f^\ast c_i(\cE)$.
\item If $0\rightarrow \cE'\rightarrow \cE\rightarrow\cE''\rightarrow 0$ is exact, then
$c_t(\cE) = c_t(\cE')\cdot c_t(\cE'')$.
\end{enumerate}
\end{prop}

The following principle allows simplified statements for the next set of definitions.
\begin{prop}[Splitting Principle]
Given $\cE$ on $X$ there exists a morphism $f:X'\rightarrow X$ such that $f^\ast:\CH(X) \rightarrow \CH(X')$ is injective
 and $\cE' = f^\ast\cE$ splits. Explicitly this means we may write
 $\cE' = \cE_0' \supseteq \cE_1' \subseteq \cdots \subseteq \cE_r' = 0$ so that the successive quotients are invertible sheaves.
\end{prop}
See \cite[Sec. 3.2 Thm. 3.2]{FultonIntersection1}.

\begin{df}
It follows from functoriality that if $\cE$ splits with quotients $\cL_1,\ldots,\cL_r$ then:
\[ c_t(\cE) = \prod_{i=1}^r c_t(\cL_i) =  \prod_{i=1}^r (1+a_it) . \]
We define the \inddef{exponential Chern character} to be:
\[ \indnotalpha{\ch(\cE)}{Chern character} = \sum_{i=1}^r e^{a_i} =  \sum_{i=1}^r \left( \sum_n \tfrac{1}{n!} a_i^n\right ). \]
We define the \inddef{Tod class} to be:
\[\indnotalpha{\td(\cE)}{tod class} = \prod_{i=1}^r \frac{a_i}{1-e^{-a_i}}, \]
where $\frac{x}{1-e^{-x}} = 1 + \frac{1}{2}x + \frac{1}{12}x^2 - \frac{1}{720}x^4 + \ldots$.
\end{df}
With this notation we can express certain functorialities in a simple manner as follows:
\begin{itemize}
\item $\ch(\cE\oplus \cF) = \ch(\cE) + \ch(\cF)$,
\item $\ch(\cE\otimes \cF) = \ch(\cE)\ch(\cF)$, and
\item $\ch(\cE^\vee) = \ch(\cE)^{-1}$.
\end{itemize}

\subsubsection{The Euler Characteristic}

\begin{thm}[Serre]
Let $X$ be a projective scheme over a Noetherian ring $A$ and let $\cO_X(1)$ be a very ample invertible sheaf on $X$ over $\Spec(A)$. Let $\cE$ be a coherent sheaf on $X$. Then the following properties hold:
\begin{enumerate}
\item For each $i\ge 0$ the $i$th cohomology $H^i(X,\cE)$ is a finitely generated $A$-module.
\item There exists an $n_0$ such that $H^i(X,\cE(n)) = 0$ for all $i>0$ and $n\ge n_0$.
\end{enumerate}
\end{thm}
See \cite[III.5.2]{Hartshorne}.

\begin{df}
Let $X$ be a projective scheme over $k$ and let $\cE$ be a coherent sheaf on $X$ we define the \inddef{Euler characteristic} of $\cE$ to be:
\[\indnotalpha{\chi(\cE)}{xeuler characteristic} = \sum_i (-1)^i\dim_k H^i(X,\cE). \]
\end{df}

\begin{prop}
Let $X$ be a projective scheme over $k$, let $\cO_X(1)$ be a very ample invertible sheaf on $X$ over $k$, and let $\cE$ be a coherent sheaf on $X$.
There exists $P(z)\in \bQ[z]$ such that $\chi(\cE(n)) = P(n)$ for all $n$.
We call $P$ the \inddef{Hilbert polynomial} of $\cE$ relative to $\cO_X(1)$.
\end{prop}
See \cite[Thm. I.7.5 and Ex. 2.7.6]{Hartshorne}.

\begin{thm}[Hirzebruch-Riemann-Roch]
For a locally free sheaf $\cE$ of rank $r$ on a non-singular projective variety $X$ of dimension $n$ we have the following formula for the Euler characteristic:
\[ \chi(\cE) = \deg(\ch(\cE).\td(\cT_X))_n. \]
\end{thm}
The statement is from \cite[A.4.1]{Hartshorne}. For the proof see
\cite{BorelSerreRR}.

\begin{cor}\label{cor:UniversalPoly}
Consider a locally free sheaf $\cE$ of rank $r$ on a smooth projective variety $X$ of dimension $n$. There exists a `universal polynomial' $Q$ such that:
\begin{align*}
 \chi(\cE) &= Q(c_1(\cE),\ldots,c_r(\cE); c_1(\Omega_X^1),\ldots, c_n(\Omega_X^1) ) \\
              &= \sum_{i=0}^n  \sum_{\abs{\alpha} = i} \sum_{\abs{\beta}=n-i} a_{\alpha,\beta}  c^{\beta}(\cE) \cdot c^{\alpha}(\Omega_X^1),
\end{align*}
where $\alpha,\beta$ are partitions of $i,n-i$, and the $a_{\alpha,\beta}$ are integers which depend only on $\alpha,\beta,n$.
\end{cor}
\begin{proof}
This follows from the observation that the Tod and Chern characters are universal polynomials in the Chern classes.
\end{proof}

\subsection{Kodaira Vanishing}

In order to effectively apply this theorem to computing dimensions of $H^0$s, one needs to know that, for the line bundle in question, the higher cohomology vanishes. To this end we have the following results.

\begin{thm}[Kodaira]
If $X$ is a non-singular projective variety of dimension $n$ and $L$ is an ample line bundle on $X$ then:
\[ H^i(X,L^{\otimes(-m)}) = 0 \text{ for all } m>0, i < n. \]
\end{thm}
The statement is \cite[Rem. III.7.15]{Hartshorne}. For the proof see \cite{KodairaVanishing}.

\begin{cor}
If $X$ is a non-singular projective variety of dimension $n$ and $L$ is an ample line bundle on $X$ then:
\[ H^i(X,L^{\otimes(m)} \otimes \Omega_X^1) = 0 \text{ for all } m>0, i > 0. \]
\end{cor}
This follows immediately from the previous result by Serre duality (see \cite[III.7 and III.7.15]{Hartshorne}).

\subsection{Hirzebruch-Proportionality}

In order to effectively apply the Riemann-Roch theorem to the situation of locally symmetric spaces there are a number of key issues that must be overcome.
The first is that one must be working with a line bundle on a projective variety.
It is not immediately apparent that modular forms should be sections of such a bundle and this should not be assumed lightly.
The second is how to actually compute the various intersection pairings that make up the Riemann-Roch formula.
Both of these problems have at least partial solutions coming out of the theory of toroidal compactifications (see \cite{AMRT,Mumford_Proportionality}).

\begin{nota}
Throughout this section we will be using the following notation.
Let $\cD = G/K$ be a Hermitian symmetric domain of the non-compact type and let $\breve{\cD} = G^c/K$ be its compact dual. 
Each of these has the induced volume form coming from the identification of tangent spaces at a base point with part of the Lie algebra $\fp_\bC \subset \fg_\bC$.

Let $\Gamma \subset \Aut(\cD)$ be a neat arithmetic subgroup with finite covolume and let $X = \Gamma \backslash \cD$ be the corresponding locally symmetric space.
We will denote by $\overline{X}$ a choice of smooth toroidal compactification and by $\overline{X}^{BB}$ the Baily-Borel compactification.
\end{nota}
\begin{df}
We then define the \inddef{Hirzebruch-Mumford volume} to be:
\[ \indnotalpha{\Vol_{HM}(X)}{Volhm} = \frac{ \Vol(X) }{ \Vol(\breve{\cD}) }. \]
\end{df}

\begin{prop}
Given a $G$-equivariant analytic vector bundle $E_0$ on $\cD$ there exists:
\begin{itemize}
\item an analytic vector bundle $\indnotalpha{\breve{E}}{Ebreve}$ on $\breve{D}$ which agrees with $E_0$ on $\cD$,
\item an analytic vector bundle $E$ on $X$ with an induced Hermitian metric, and
\item a unique extension $\indnotalpha{\overline{E}}{Eoverline}$ to $\overline{X}$ such that the induced metric is a good singular metric on $\overline{X}$.
\end{itemize}
\end{prop}
See \cite[Thm 3.1]{Mumford_Proportionality}.

\begin{thm}
Using the notation of the previous proposition.
For each partition $\alpha$ of $n=\dim(X)$ the associated Chern numbers $c^{\alpha}(\overline{E})$ and $c^{\alpha}(\breve{E})$ satisfy the following relation:
\[ c^{\alpha}(\breve{E}) = (-1)^{\dim(X)} \Vol_{HM}(X) c^{\alpha}(\overline{E}). \]
\end{thm}
See \cite[Thm 3.2]{Mumford_Proportionality}.

\subsubsection{Geometric Modular Forms}

We now give a definition of the spaces in which we are interested.

\begin{df}
Given a representation $\rho: K \rightarrow \GL_n$ we define a bundle $\indnotalpha{E_\rho}{Erho}$ on $\cD$ via 
\[  E_\rho = K\backslash (G \times_\rho \bC^n). \]
We define a \inddefssalpha{$\rho$-form}{modular form}{rho-form} on $X$ to be a $\Gamma$-equivariant section of $E_\rho$ such that the induced map $\tilde{f}:G \rightarrow \bC^n$ satisfies:
\[ \abs{\tilde{f}(g)} \leq C \abs{\abs{g}}_G^n \]
for some $n>1,C>0$. The norm $\abs{\abs{g}}$ is defined as in \cite[Sec. 7]{BorelAutomorphic} as $\Tr(\Ad(s(g))^{-1}\cdot \Ad(g))$, where $s$ is a Cartan involution.

We say a $\rho$-form is \inddefns{holomorphic}{modular form} if it is a holomorphic section of:
\[ \breve{E}_\rho = K_\bC P_+ \backslash (G_\bC \times_\rho \bC^n) \]
on the inclusion of $E \injects \breve{E}$.
\end{df}

\begin{prop}
The vector space of holomorphic $\rho$-forms is precisely:
\[ H^0(\overline{X},\overline{E_\rho}), \]
where $\overline{X}$ is a smooth toroidal compactification of $X$ and $\overline{E}_\rho$ the unique extension of $E_\rho$ to $\overline{X}$.
\end{prop}
See \cite[Prop 3.3]{Mumford_Proportionality}.

\begin{prop}
Consider the case $\breve{E} = \Omega^1_{\breve{\cD}}$ so that $E=\Omega^1_{\cD}$. In this case
\[ \overline{E} = \indnotalpha{\Omega^1_{\overline{X}}(\log)}{Omegaonelog} \]
is the bundle whose sections near a boundary of $k$ intersecting hyperplanes are of the form:
\[ \sum_{i=1}^k a_i(z)\tfrac{dz_i}{z_i} + \sum_{i=k+1}^n a_i(z)dz_i. \]
\end{prop}
See \cite[Prop 3.4.a]{Mumford_Proportionality}.

\begin{prop}
Consider the case $\breve{E} = \Omega^n_{\breve{\cD}}$ so that $E=\Omega^n_{\cD}$ is the canonical bundle of $\cD$. In this case
\[ \overline{E} = f^\ast(\cO_{\overline{X}^{BB}}(1)) \]
is the pullback of an ample line bundle on the Baily-Borel compactification.
The sections of $\cO_{\overline{X}^{BB}}(n)$ are the modular forms of weight $n$.
\end{prop}
See \cite[Prop 3.4.b]{Mumford_Proportionality}.

\begin{cor}\label{cor:AvoidBoundary}
Suppose $n' = \dim(\overline{X}^{BB} - X)$, then for all $k > n'$ the cycle $[\Omega_{\overline{X}}^1(\log)]^k$ is supported on $X$.
\end{cor}
\begin{proof}
This is true for the ample line bundle on $\overline{X}^{BB}$ for which $\Omega_{\overline{X}}^1(\log)^k$ is the pull back. Hence the statement is true for $\Omega_{\overline{X}}^1(\log)^k$.
\end{proof}

\begin{cor}
For $X = \Gamma\backslash \cD$ a locally symmetric space, the modular forms are:
\[ \indnota{M_k(\Gamma)} = H^0(\overline{X}, \Omega^n_{\overline{X}}(\log)^k) \]
is the space of modular forms of weight $k$ level $\Gamma$ for $G$. Furthermore the cusp forms are:
\[ \indnota{S_k(\Gamma)} = H^0(\overline{X}, \Omega^n_{\overline{X}}(\log)^{k-1}\otimes \Omega^n_{\overline{X}}). \]
\end{cor}

\subsubsection{Computing Dimensions}

We now describe how to compute dimensions for spaces of modular forms.

\begin{prop} \label{prop:SimpleSupport}
Suppose $D$ is a cycle on $\overline{X}$ supported entirely on $X$, then  
\[ D \cdot c^\alpha(\Omega_{\overline{X}}(\log)) = D \cdot c^\alpha(\Omega_{\overline{X}}). \]
\end{prop}
This follows from the properties of the Chern classes.

\begin{lemma}
Suppose $Q$ is the universal polynomial of Corollary \ref{cor:UniversalPoly}
then:
\begin{align*} E_{\overline{X}}(\ell) :&= \begin{aligned}[t] Q(\ell c_1(\Omega_{\overline{X}}^1(\log));& c_1(\Omega_{\overline{X}}^1(\log)),\ldots,c_n(\Omega_{\overline{X}}^1(\log)))\\& - Q(\ell c_1(\Omega_{\overline{X}}^1(\log)); c_1(\Omega_{\overline{X}}^1),\ldots,c_n(\Omega_{\overline{X}}^1))\end{aligned} \\ &= 
   \sum_{i=0}^{n'} \ell^i [c_1(\Omega_{\overline{X}}^1(\log))^i] \sum_{\abs{\alpha}=n-i }b_{\alpha} (c^\alpha(\Omega_{\overline{X}}^1)-c^\alpha(\Omega_{\overline{X}}^1(\log)))
\end{align*}
for constants $b_\alpha$ which depend only on $\alpha$ and not on $X$.
\end{lemma}
\begin{proof}
This is a direct application of Corollary \ref{cor:AvoidBoundary} and Proposition \ref{prop:SimpleSupport}.
\end{proof}

\begin{thm}\label{thm:Proportionality}
Consider $(\Omega^n_{\breve{\cD}})^{-1}$ the ample line bundle on $\breve{\cD}$ and let 
\[ P_{\breve{\cD}}(\ell) = \sum_i \dim(H^i(\breve{\cD},(\Omega^n_{\breve{\cD}})^{-1})) \]
be the associated Hilbert polynomial.
Suppose $\Gamma$ is a neat arithmetic subgroup and $\overline{X}$ is a smooth toroidal compactification of $X = \Gamma \backslash \cD$ with $n' = \dim(\overline{X}^{BB} - X)$.
Then for $\ell \ge 2$ we have:
\[ \dim(S_\ell(\Gamma)) = \Vol_{HM}(X) P_{\breve{\cD}}(\ell-1) - E_{\overline{X}}(\ell).\]
\end{thm}
See \cite[Prop 3.5]{Mumford_Proportionality}.
\begin{rmk}
A remark is in order on the issue of the \textit{weight} of a modular form. 
The \textit{weight} $\ell$ in the above theorem is what is known as the geometric weight. This differs from the arithmetic weight by a factor of $\dim(X)$.
\end{rmk}

\begin{nota}
Denote the boundary of $X$ by $\Delta = \overline{X} - X$ and write $[\Delta] = \sum [D_i]$ as a decomposition into its irreducible components $[D_i]$.
Denote by $\Delta_k$ the $k$th elementary symmetric polynomial in the $[D_i]$.
Moreover, for $\alpha$ a partition denote by $\indnotalpha{\Delta^{\alpha}}{Deltaalpha} = \prod_\ell \Delta_{\alpha_\ell}$.
\end{nota}

\begin{prop}
Let $\overline{X}$ be an $n$ dimensional complex manifold and suppose $\Delta=\overline{X}\setminus X$ is a reduced normal crossings divisor.
Denoting by $\Omega_{\overline{X}}^1(\log)$ the subsheaf of $\Omega_{\overline{X}}$ with $\log$-growth near $\Delta$.
Then:
\[ c_j(\Omega^1_{\overline{X}}) = \sum_{i=0}^j c_{i}(\Omega^1_{\overline{X}}(\log))\Delta_{j-i}. \]
\end{prop}
\begin{proof}
This is proven is slightly more generality in \cite[Prop 1.2]{Tsushima} for the tangent bundle.
It follows from considering the following two exact sequences:
\[ \xymatrix{
0 \ar[r]& \Omega^1_{\overline{X}}(\log) \ar[r]& \Omega^1_{\overline{X}} \ar[r]& \oplus \cO_{D_i}(D_i) \ar[r]& 0, \\
0 \ar[r]& \cO_{\overline{X}} \ar[r]& \cO_{\overline{X}}(D_i) \ar[r]& \cO_{D_i}(D_i) \ar[r]& 0. 
}\]
\end{proof}

\begin{cor}
For a partition $\alpha$ of $j$ we find:
\begin{align*}
  c^\alpha(\Omega^1_{\overline{X}}) 
                      &= \prod_{\ell} \left(\sum_{i=0}^{\alpha_\ell} c_{i}(\Omega^1_{\overline{X}}(\log))\Delta_{\alpha_\ell-i} \right) \\
                      &= \sum_{\beta,\gamma} d_{\alpha,\beta,\gamma} c^\beta(\Omega^1_{\overline{X}}(\log))\Delta^{\gamma}, 
\end{align*}
where the $d_{\alpha,\beta,\gamma}$ depend only on $\alpha,\beta,\gamma$ and not on $X$.
\end{cor}

\begin{cor}
We have that:
\begin{align*}
 E_{\overline{X}}(\ell) &= 
\sum_{i=0}^{n'} \ell^i [c_1(\Omega_{\overline{X}}^1(\log))^i] 
          \sum_{\abs{\alpha}=n-i }b_{\alpha} 
                   \left(\underset{\abs{\gamma}=\abs{\alpha}-\abs{\beta}}{\underset
{\abs{\beta}<\abs{\alpha}}\sum} d_{\alpha,\beta,\gamma}  
                              c^\beta(\Omega^1_{\overline{X}}(\log))\Delta^{\gamma} \right),
\end{align*}
where the coefficients $b_{\alpha}$ and $d_{\alpha,\beta,\gamma}$ depend only on $\alpha,\beta,\gamma$ and $n$ and not otherwise on $X$.
\end{cor}

It is possible to obtain an even more precise formula for $E_{\overline{X}}(\ell)$.
In particular the following theorem from \cite[Thm. 4.1]{fiori5} tells us that the error term is a formal combination of Euler characteristics for various intersections of boundary components.
\begin{thm}
Let $X$ be a smooth projective variety and let $\cF$ be any coherent sheaf on $X$.
Suppose $\Delta$ is a collection of smooth divisors with simple normal crossings on $X$.
Then
\[
 \chi(X,\cF) - \chi(X,\Delta,\cF)  =  \sum_{\abs{\underline{b}}\ge 1} (-1)^{\abs{\underline{b}}} c_{\underline{b}} D^{\underline{b}} Q_{n-\abs{\underline{b}}}(\ch_j(\cF); \ch_i(\Omega_X^1(\log(\Delta)))) .
\]
\end{thm}

\begin{rmk} We have the following remarks about the above:
\begin{itemize}
\item
All of the intersections in the above formula take place in the boundary, since $\abs{\gamma} > 0$  for every term appearing in the formula. 
\item 
There are only finitely many connected components of boundary components and finitely many inequivalent orbits of boundary component.
\item
Boundary components are of the form:
\[ \overline{\Gamma_F\backslash F}\ltimes (\bZ^{2m}\backslash\bC^m) \ltimes \overline{O(\sigma)} \]
for the various boundary components $F$ and cones $\sigma$.
\item Intersections between adjacent $F$'s in $\overline{X}^{BB}$ is understood by the spherical Bruhat-Tits building of $G$ over $\bQ$.
\item The intersections of two cones in $F$ are either another cone of $F$ or a cone of an adjacent boundary component $F'$ contained in the closure of $F$.
\item We can select a toroidal compactification where the boundary has simple normal crossings.
\end{itemize}
\end{rmk}

\begin{rmk}
The above results combine to reduce the issue of computing dimension formulas to the following steps:
\begin{enumerate}
\item Computing the Hilbert polynomial $P_{\breve{\cD}}$.
            These are known in all the basic cases.
\item Computing the volume $\Vol_{HM}(X)$.
           This depends on the choice of $\Gamma$, the formulas typically involve special values of $L$-functions.
\item Understanding the arithmetic of the group $\Gamma$ well enough to describe all of the cusps.
\item Understanding the geometry of the toroidal compactifications over the cusps. 
\item Computing the relevant Euler characteristics of these pieces.

On this point it is worth noting that the sheaves under consideration are the pullbacks of sheaves of modular forms on the Baily-Borel boundary components. 
\end{enumerate}
\end{rmk}

\subsection{The Orthogonal Case}
\label{sec:SectionDimensionOrthogonal}

The following discussion follows closely that of \cite[Section 2]{GHSOrthogonal}.

\begin{thm}
Let $\cD$ be the symmetric space for an orthogonal group of signature $(2,n)$, then:
\[       \chi(\cO_{\breve{\cD}}(-n)^\ell) =  \chi(\cO_{\bP^{n+1}}(-n\ell)) -  \chi(\cO_{\bP^{n+1}}(-n\ell-2) ) =  \left(\begin{smallmatrix} n+1-n\ell \\ n \end{smallmatrix}\right) -  \left(\begin{smallmatrix} n-1-n\ell \\ n \end{smallmatrix}\right).\]
\end{thm}
\begin{proof}
We have describe $\breve{\cD}$ as a quartic in $\bP^{n+1}$ with canonical bundle $\cO_{\breve{\cD}}(-n)$. The adjunction formula places it into the following exact sequence:
\[ 0 \rightarrow \cO_{\bP^{n+1}}(-n\ell - 2) \rightarrow \cO_{\bP^{n+1}}(-n\ell) \rightarrow \cO_{\breve{\cD}}^\ell \rightarrow 0. \]
This allows us to compute the Hilbert polynomial of $\cO_{\breve{\cD}}$ from that of $\cO_{\bP^{n+1}}$.
In particular using the fact that $\dim(H^0(\cO_{\bP^{n+1}}(k))) = \left(\begin{smallmatrix} n+1+k \\ n \end{smallmatrix}\right)$ allows us to check the result.
\end{proof}

The non-trivial volume forms on a Hermitian symmetric domain $\cD$ are induced by the Killing form and the identification of $\fp$ with $\cT_{\cD,x}$, where $x$ is any base point.
Up to scaling this form is unique.

For the group $\Orth_{2,n}$ it is shown in \cite[p. 239]{Helgason} that the tangent spaces for $\cD$ and $\breve{\cD}$ are respectively:
\[ \mtx 0U{U^t}0 \quad \text{and}\quad \mtx 0U{-U^t}0 \]
in the Lie algebra of $G$.
The killing form is $\Tr(M_1 M_2^t)$ which induces the form $2\Tr(U_1U_2^t)$.
Fix a lattice $L$ in the underlying quadratic space.
In \cite{SiegelQF} Siegel computed the volume of $\Orth(L) \backslash \cD$ relative to $\Tr(U_1U_2^t)$ as:
\[  2\alpha_\infty(L,L)\abs{D(L)}^{(2+n+1)/2}\left(\prod_{k=1}^2 \pi^{-k/2}\Gamma(k/2)\right)\left(\prod_{k=1}^n \pi^{-k/2}\Gamma(k/2)\right), \]
where $\alpha_\infty(L,L)$ is the real Tamagawa volume of $\Orth(L)$.
The computations of \cite{HuaVolumes} when combined with the above yield the formula:
\[ \Vol(\breve{\cD}) = 2 
\left(\prod_{k=1}^{n+2} \pi^{k/2}\Gamma(k/2)^{-1}\right) 
\left(\prod_{k=1}^n \pi^{-k/2}\Gamma(k/2)\right) 
\left(\prod_{k=1}^{2} \pi^{-k/2}\Gamma(k/2)\right). \]
Combining these results we find:
\begin{prop}\label{prop:HMVolOrth}
The Hirzebruch-Mumford volume for an orthogonal symmetric space is: 
\[ \Vol_{HM}(\SO(L)\backslash \cD) = \alpha_\infty(L,L)\abs{D(L)}^{(2+n+1)/2}\left(\prod_{k=1}^{n+2} \pi^{k/2}\Gamma(-k/2)\right). \]
\end{prop}
In order to compute $\alpha_\infty(L,L)$ we use several facts.
\begin{prop}
For an indefinite lattice of rank at least $3$ the genus equals the spinor genus.
\end{prop}
This follows from \cite[Thm 6.3.2]{Kitaoka}.

\begin{prop}
The weight of a lattice depends only on its spinor genus.
\end{prop}
This is discussed in \cite[p224]{GHSOrthogonal}. See also \cite[Thm 5.10]{ShimuraExact}.

Now using the fact that the Tamagawa volume of $\SO_V(\bQ) \backslash \SO_V(\bA) = 2$ we may conclude:
\begin{prop}\label{prop:HMVolOrthLocal}
For an indefinite lattice of rank at least $3$ the following formula holds:
\[ \prod_p \alpha_p(L,L) = \frac{2}{\abs{spn^+(L)}} \]
or equivalently:
\[ \alpha_\infty(L,L) = \frac{2}{\abs{spn^+(L)}} \prod_p \alpha_p(L,L)^{-1}, \]
where $spn^+(L)$ is the proper spinor genus of $L$.
\end{prop}

\begin{rmk}
It is known (see \cite[Cor 6.3.1]{Kitaoka}) that $\abs{spn^+(L)}$ is a power of $2$. Moreover, by \cite[Cor 6.3.2]{Kitaoka} computing  $\abs{spn^+(L)}$ can be reduced to a finite computation.

The local densities $\alpha_p(L,L)$ can also be computed.
These computations are explained in \cite[Ch. 5]{Kitaoka}.
Note that $\alpha_p$ differs from $\beta_p$ by a factor of $q^{\rank(L)\nu(2)}$.

\end{rmk}

\begin{rmk}
There is an important remark to be made on the subject of computing the Euler-characteristics of the boundary in this case.

The theory of Generalized Heegner divisors gives a natural method of computing explicit representatives for the self intersection terms which are necissary (see \cite[Sec. 6]{fiori5} for a discussion of how to handle the self intersection terms in general). 

The boundary components are either toric varieties, whose Euler characteristics are purely combinatorial, or compactifications of $k$-fold fiber products of universal elliptic curves.
The Euler characteristics of these should be related to dimensions of spaces of modular forms on the underlying curves.
\end{rmk}

}

{\def\XMetaCompile{1}
}

{\def\XMetaCompile{1}

\subsection{Non-Neat Level Subgroups}

An important aspect of the above discussion was the appearance of the term `non-singular'. In order to obtain a non-singular variety from a locally symmetric space one is forced to take blowups. This process is not (trivially) well-behaved with respect to the existence or dimension of sections. The above machinery only works directly, without the need for any modifications, when the locally symmetric space is non-singular.
Consequently, an important result is that every locally symmetric space has a non-singular finite cover. This result follows from the following:
\begin{thm}
Suppose $p \nmid \Phi_\ell(1)$ and $\deg(\Phi_{\ell})\leq n$ for all $\ell,$ then $\Gamma(p) \subset \Gl_n(\bZ)$ is neat.
\end{thm}
See \cite[Prop. 17.4]{BorelArithmetic}.

Two natural questions now arise:
\begin{qu}
What does it mean to have a modular form on a singular space?
\end{qu}
\begin{qu}
How can one compute the dimension of this space from the corresponding dimension of the cover?
\end{qu}
\begin{rmk}
The reason the first question is important is that line bundles may not descend to a desingularization of the quotient.
Notice that the desingularization of $\overline{(\SL_2(\bZ) \backslash \uhp)}$ is $\bP^1$. If the line bundle of modular forms of weight $2$ descended, it would by necessity have global sections.
Moreover, even if the line bundle does descend, it is not clear that $\Gamma$-invariant sections will descend to holomorphic sections.
\end{rmk}

\begin{nota}
Suppose we have a normal subgroup $\Gamma'\subset \Gamma$ with $\Gamma'$ neat.
Denote by $S_k(\Gamma')$ the space of weight $k$ cusp forms on $X(\Gamma').$
Define $\indnota{S_k(\Gamma)} = S_k(\Gamma')^{\Gamma}$ to be the space of $\Gamma$-invariant cusp forms.
Define $\indnotalpha{\tilde{S}_k(\Gamma)}{Stildek}\subset {S}_k(\Gamma)$ to be the subspace of cusp forms which extend to holomorphic forms on a desingularization $\tilde{X}(\Gamma)$ of $X(\Gamma) = \Gamma\backslash X(\Gamma')$.
\end{nota}

\begin{prop}
With the notation as above we can compute:
\[ \dim(S_k(\Gamma)) = \sum_{\gamma\in\Gamma/\Gamma'} \tr(\gamma | S_k(\Gamma')). \]
\end{prop}
The proof is a standard argument.
A generalization of the Riemann-Roch theorem by Atiyah and Singer \cite{AtiyahSinger3} allows this to be computed.

We first introduce the following notation:
\begin{nota}
Suppose $\gamma\in \Gamma$, $\chi$ is a character of $\Gamma$ and $\theta\in \bC^\times$.
Denote by $X^\gamma = \{ x\in X \mid x=\gamma(x) \}$ and by $N_\gamma = N_{X_\gamma}$ the normal bundle of $X^\gamma$ in $X$.
For a vector bundle $\cE$ denote by $\cE_\gamma(\theta)$ the $\theta$-eigenspace of $\gamma$ and by $\cE(\chi)$ the $\chi$-isotypic component.
Suppose $c_t(\cE) = \prod(1-x_it)$, then set $U^{\theta}(\cE) = \prod(\frac{1-\theta}{1-\theta e^{x_i}})$ and $\ch(\cE)(\gamma) = \sum_\chi \chi(\gamma)\ch(\cE(\chi))$.
\end{nota}
 
\begin{thm}
Suppose $k$ is sufficiently large so that $H^i(\overline{X}, \Omega_X^N(\log)^{k-1})=0$ for $i>0$ then:
\[ \tr(\gamma | S_k(\Gamma) ) = \left\{ \frac{ \ch( \Omega_X^N(\log)^{k-1}\otimes \Omega_X^N | X^\gamma)(\gamma) \prod_\theta U^\theta( N_\gamma(\theta) ) \td(X^\gamma)}{\det(1-\gamma|N_\gamma^\ast)}\right\} [X^\gamma]. \]
This is a polynomial in the weight $k$ of degree at most $X^\gamma$.
\end{thm}
See \cite[Sec. 2]{Tai_Kodaira} and \cite[Thm. 3.9]{AtiyahSinger3}.

\begin{rmk}
The contribution of the identity element of $\Gamma$ in this formula gives us the Riemann-Roch theorem for $S_k(\Gamma)$.
To evaluate this formula one needs a complete understanding of the ramification locus of the quotient map.
\end{rmk}

On the issue of the relation of $S_k(\Gamma)$ to $\tilde{S}_k(\Gamma)$ we have the following result.
\begin{prop}
Let $\tilde{X}(\Gamma)$ be a non-singular model of $X(\Gamma)$ and let $\tilde{X}(\langle\gamma,\Gamma' \rangle)$ be the non-singular model of $X(\langle\gamma,\Gamma' \rangle)$ which covers it.
A $\Gamma'$-invariant form extends to $\tilde{X}(\Gamma)$ if and only if it extends to $\tilde{X}(\langle\gamma,\Gamma' \rangle)$ for all $\gamma\in \Gamma$.
\end{prop}
See \cite[Prop. 3.1]{Tai_Kodaira}.

\begin{df}
Let $\gamma$ act on $X$ with a fixed point $x\in X$. 
Suppose the eigenvalues for the action of $\gamma$ on $\cT_{X,x}$ are $e^{2\pi i \alpha_j}$ for $j=1,\ldots,n$.
We say the singularity at $x$ is \inddefsalpha{$\gamma$-canonical}{singularity}{gamma-canonical} if $\sum_j \alpha_j -\floor{\alpha_j} \ge 1$.
\end{df}

\begin{prop}
Every invariant form extends to $\tilde{X}(\langle\gamma,\Gamma \rangle)$ if and only if all the singularities are $\gamma^k$-canonical for all $\gamma^k\neq \Id$. 
\end{prop}
See \cite[Prop. 3.2]{Tai_Kodaira}.
\begin{rmk}
Forms which have sufficiently high orders of vanishing along the ramification divisor will still extend even if the singularities are not canonical.
\end{rmk}

\begin{thm}
Let $L$ be a lattice of signature $(2,n)$ with $n\ge 9$ and let $\Gamma \subset \Gamma'$ be as above.
There exists a toroidal compactification of $X(\Gamma)$ such that all the singularities are $\gamma$-canonical for all $\gamma \in \Gamma'$.
\end{thm}
See \cite[Thm 2]{GHSKodairaK3}.
\begin{rmk}
The results of \cite{GHSKodairaK3} are slightly more refined. They show that for $n\ge 6$ the only source of non-canonical singularities on the interior are reflections. For $n\ge 7$ the reflections no longer give non-canonical singularities. For the boundary, they show the $0$-dimensional cusps never present non-canonical singularities (by a choice of toroidal compactification). They also show that the $1$-dimensional cusps may only have non-canonical singularities over the usual points $i,\omega\in\uhp$ and these points present no problems if $n\ge 9$.
Moreover, from their proof one can compute lower bounds on $\ell$ such that $\Gamma(\ell)$ would only give canonical singularities.

The computations involved in obtaining these results use the structure of singularities that we will discuss in the following section.
\end{rmk}

\section{Ramification for Orthogonal Shimura Varieties}
\label{sec:SectionRamification}

The purpose of this section is to describe the nature of the ramification between different levels for the orthogonal group. The only other discussion of this topic with which we are familiar is the work of \cite[Sec. 2]{GHSKodairaK3}. Some of the results here are motivated by their constructions.

Let $L$ be a $\bZ$-lattice of signature $(2,n)$.
Recall that:
\[ \cD_L=\cK_L= \{ [\vec{z}]\in\bP(L\otimes_\bZ\bC) \mid q(\vec{z}) = 0,\; b(\vec{z},\overline{\vec{z}}) >0  \}. \]
Denote by $\Orth_L$ the orthogonal group of $L$.
For $\Gamma$ a subgroup of $\Orth_L(\bZ)$ we set:
\[ X_L(\Gamma) := \Gamma\backslash \cD_L. \]
When $\Gamma$ is neat $X_L(\Gamma)$ can be given the structure of a smooth quasi-projective variety.
We also wish to think about $X_L(\Gamma)$ when $\Gamma$ is not neat. It will be a quotient of $X_L(\Gamma')$ for some neat subgroup $\Gamma' \subset \Gamma$ by a finite group of automorphisms.
The quotient certainly exists as a stack (though we shall not discuss this further). However, one often expects that one can make sense of it as a scheme, in which case the cover $\pi_\Gamma:X_L(\Gamma')\rightarrow X_L(\Gamma)$ will be a ramified covering.

The first thing we shall do is describe the structure of some `explicit' ramification divisors. We will next explain why this captures all of the ramification.

\subsection{Generalized Heegner Cycles}

We now define a class of cycles on our spaces. This is essentially the same definition as the cycles considered in \cite{kudlamsri}, see also \cite{KudlaCycles}.
\begin{df}
Let $S\subset L$ be a (primitive) sublattice of signature $(2,n')$.
Then $S^\perp$ is a (primitive) negative-definite sublattice of $L$. Define:
 \[ \indnotalpha{\cD_{L,S}}{DcLS} = \{ [\vec{z}] \in \cD_L \mid b(\vec{z}, \vec{y}) = 0 \text{ for all } \vec{y}\in S^\perp \}. \]
This is a codimension $\rank{S^\perp}$ subspace of $\cD_L$, defined by algebraic conditions. Moreover, we see that:
\[ \cD_S \simeq \cD_{L,S} \subset \cD_L. \]

Let $\Phi_S = \{ S' \mid S' = \gamma S \text{ for some } \gamma \in \Gamma\}$. Define:
\[ \indnota{H_{L,S}} = \underset{S'\in \Phi_S}\cup \cD_{L,S'} \]
 to be the \inddef{generalized Heegner cycle} associated to this set of (primitive) embeddings of $S$ into $L$.
Its image in $X_L(\Gamma)$ will be an analytic cycle. A more careful analysis and a precise definition can result in obtaining an algebraic cycle (see \cite{kudlamsri}).
\end{df}
\begin{rmk}
In the definitions above we could just as well have taken $S\subset L^\#$, the dual of $L$, or in fact any lattice in $L\otimes \bQ$. However, for our purposes, since $(S^\perp)^\perp \cap L$ would give a primitive lattice generating the same $\cD_{L,S}$, there is no real loss of generality in assuming this for our purposes.

We should remark that if $S$ has corank $1$ then $H_{L,S} = H_{\overline{x}_i,q(x_i)}$ is just a usual Heegner divisor (see \cite[p. 80]{Brunier_BP}). This justifies our choice of name. It is not our intent to imply that there is (or is not) a relation to the generalized Heegner cycles arising from certain Kuga-Sato varieties (see \cite{Bertolini10chow-heegnerpoints}).
\end{rmk}

\subsection{Ramification near $\cD_{L,S}$}
\label{sec:SectionRNGHS}

We introduce the following notation (for any non-degenerate $S$):
\begin{align*}
 \Gamma_S &= \{ \gamma\in \Gamma \mid \gamma S \subset S \}, \\
 \overline{\Gamma}_S &= \{ \gamma\in \Orth_S \mid \gamma \text{ lifts to } \Gamma \}, \text{ and} \\
 \tilde{\Gamma}_S &= \{ \gamma\in \Gamma_S \mid \gamma|_{S^\perp} = \Id\}.
\end{align*}
\begin{rmk}
It would be convenient if $\tilde{\Gamma}_S \simeq \overline{\Gamma}_S$, however, this is hard to guarantee if $L \neq S \oplus S^\perp$.
\end{rmk}
We return to the setting where $S\subset L$  is a sublattice of signature $(2,n')$, so that $S^\perp$ is a negative-definite lattice.
It follows that $\overline{\Gamma}_{S^\perp}$, and hence $\tilde{\Gamma}_{S^\perp}$, are both finite groups.
We find that $\tilde{\Gamma}_S \times \tilde{\Gamma}_{S^\perp} \injects \Orth_L$, while $ \overline{\Gamma}_S \times \overline{\Gamma}_{S^\perp}$ may not.
We have the following maps:
\[ \xymatrix{ X_S(\tilde{\Gamma}_S)  \ar@{->>}[d] \ar@{^{(}->}[r] &  (\tilde{\Gamma}_S\times \tilde{\Gamma}_{S^\perp})\backslash \cD_L  \ar@{->>}[d] \\
 X_S(\overline{\Gamma}_S) \ar@{->}[r] & X_L(\Gamma).
 }
\]
\begin{rmk}
If we want the bottom map to be injective we would need that for each $\sigma \in O_L$ with $x,\sigma(x) \in \cD_{L,S^\perp}$ then there exists $\tau \in O_S$ with $\tau(x)=\sigma(x)$.
\end{rmk}

We wish to explain the local ramification near $D_{L,S}$.
Fix $e_1$ and $e_2$ isotropic vectors spanning a hyperplane in $S\otimes K$, where $K$ is a totally real quadratic extension of $\bQ$. Note that we cannot always take $e_1$ and $e_2$ in $S$.
We may then choose to express the spaces $\cD_{S}$ and $\cD_{L}$ as tube domains relative to the same pair $e_1,e_2$.
In particular we may write:
\[ \cD_L = \{ \vec{u} \in \cU_L = \langle e_1,e_2\rangle^\perp \subset L\otimes \bC \mid q(\Im(\vec{u})) > 0 \}\]
with $\cD_{L,S}$ in $\cD_L$ being precisely:
\[ \cD_{L,S} = \{ \vec{u} \in \cU_S = \langle e_1,e_2, S^\perp \rangle^\perp \subset L\otimes \bC \mid q(\Im(\vec{u})) > 0 \}.\]
Thus we see that in a neighbourhood of $\cD_{L,S}$ in $\cD_L$ we can express
\[ \cD_L = \cD_{L,S} \oplus (S^\perp \otimes \bC). \]
Then $\tilde{\Gamma}_{S^\perp}$ acts on the complementary space $S^\perp \otimes \bC$.
We see that the cycle $\cD_{L,S}$ is the generic ramification locus for this action. That is, $\cD_{L,S}$ is maximal among cycles fixed by this action (with respect to inclusion among cycles).

\begin{rmk}
We remark that for some points of $\cD_{L,S}$ the group $\Gamma_{S^\perp} = \Gamma_{S}$ may also cause ramification in the quotient. This ramification will not in general be generic, and it will typically restrict to some sub-cycle of $\cD_{L,S}$.

Indeed, a group element $g$ fixes $\tau \in \cD_{L,S}$ if and only if $\tau$ is an eigenspace of $g$. Thus $g$ can only fix all of $\cD_{L,S}$ if $S$ is an eigenspace. This would imply that $\tau$ acts as $-1$ on $S$. Such an element acts trivially on $\cD_S$ as this is a projective space. The effect of the quotient by $g$ is the same as by $-g \in \Orth_{S^\perp}$.
\end{rmk}

\subsection{Generalized Special Cycles}
\label{sec:secgenspeccylc}

We will now introduce another type of cycle on the spaces $X = \Gamma\backslash \cD_L$ which play a role in ramification.
We will call these \inddef{generalized special cycles} because of their relationship to special points.

Let $F/\bQ$ be a CM-field and consider the CM-algebra:
\[ E=F^d= F^{(1)}\times\cdots\times F^{(n)}. \] 
Denote complex conjugation for both $F$ and $E$ by $\sigma$ . View $E$ as an $F$-algebra under the diagonal embedding of $F$ into $E$.
Label the embeddings $\Hom(F,\bC)$ as $\{ \rho_1,\overline{\rho_1},\ldots \rho_m,\overline{\rho_m} \}$.
Pick $\lambda = (\lambda^{(1)},\ldots,\lambda^{(n)}) \in (E^\sigma)^\times$ such that $\rho_1(\lambda^{(1)}) \in \bR^+$ but $\rho_j(\lambda^{(i)}) \in \bR^-$ for all other combinations of $i,j$.
We now consider the rational quadratic space $(V,q_{E,\sigma,\lambda})$ given by $V=E$ and
\[ q_{E,\sigma,\lambda}(x) = \tfrac{1}{2}\Tr_{E/\bQ}(\lambda x\sigma(x)). \]
Notice that the  signature of the quadratic form is of the shape $(2,\ell)$. We define also the $F$-quadratic space $(V',q_{E,\sigma,\lambda}')$ given by $V'=E$ and
\[ q_{E,\sigma,\lambda}'(x) = \tfrac{1}{2}\Tr_{E/F}(\lambda x\sigma(x)). \]
Notice that $q_{E,\sigma,\lambda}(x) = \Tr_{F/\bQ}(q_{E,\sigma,\lambda}'(x))$.
We have the tori $T_{E,\sigma}$ and $T_{F,\sigma}$ defined by:
\begin{align*} T_{E,\sigma}(R) &= \{ x\in (E\otimes_\bQ R)^\times \mid x\sigma(x)=1 \},\\ T_{F,\sigma}(R) &= \{ x\in (F\otimes_\bQ R)^\times \mid x\sigma(x)=1 \}, \end{align*} 
as well as maps:
\[  T_{F,\sigma} \overset{\Delta}\injects T_{E,\sigma} \injects \Res_{F/\bQ}(\Orth_{q_{E,\sigma,\lambda}'}) \injects \Orth_{q_{E,\sigma,\lambda}}, \]
where the first map $\Delta$ is the diagonal embedding.
Now suppose further that: $q = q_{E,\sigma,\lambda} \oplus q^\perp$ and consider the inclusion:
\[ \Orth_{q_{E,\sigma,\lambda}} \injects \Orth_q. \]

\begin{df}
The \inddef{generalized special cycle} associated to the inclusions $T_{F,\sigma} \overset{\phi}{\injects \cdots \injects} \Orth_{q}$ as above is:
\[ \indnotalpha{\cD_{\phi}}{Dcphi} = \{ [\vec{z}] \in \kappa_q^+ \mid g\vec{z} = \rho_0(g)\vec{z} \text{ for all } g\in  T_{F,\sigma}(\bR) \}. \]
For any lattice $L$ in the quadratic space of $q$ this gives us a cycle in $\cD_L$. 
Set $\Phi = \{ \gamma^{-1}\phi\gamma \mid \gamma\in \Gamma\}$ and define:
\[ \indnota{H_\phi} = \underset{\phi\in \Phi}\cup \cD_{\phi}. \]
The image of $H_\phi$ in $X=\Gamma\backslash \cD_L$ is a cycle on $X$ of the form:
\[ \Gamma'\backslash  \cD_{\phi} = \Gamma'\backslash \Res_{F/\bQ}(\Orth_{q_{E,\sigma,\lambda}'})(\bR)/K_{E,\sigma,\lambda}, \]
where $\Gamma' = \Gamma \cap \Res_{F/\bQ}(\Orth_{q_{E,\sigma,\lambda}'})(\bZ)$ and $K_{E,\sigma,\lambda}$ is a maximal compact subgroup of $\Res_{F/\bQ}(\Orth_{q_{E,\sigma,\lambda}'})(\bR)$.
Note that: 
\[ \Res_{F/\bQ}(\Orth_{q_{E,\sigma,\lambda}'})(\bR) \simeq \Orth_{2,m-2}(\bR)\times \Orth_m(\bR)^{d-1}. \]
\end{df}
\begin{rmk}
If $d=1$ then the special cycle will be a special point.
\end{rmk}

\subsection{Ramification Near $\cD_{\phi}$}

\begin{nota}
Denote the group of $N^{th}$ roots of unity by $\mu_N$ and a choice of generator by $\zeta_N$.
\end{nota}

The group $\mu_N$ has a unique irreducible rational representation $\psi_N$.
The representation $\psi_N$ is precisely the $\varphi(N)$-dimensional representation of $\mu_N$ acting on the rational vector space $\bQ(\zeta_N)$ by multiplication.

For each $a\in (\bZ/N\bZ)^\times$ the generator $\zeta_N$ acts on:
\[ x_a = \sum_{b\in\bZ/N\bZ} \zeta_N^{-b}\otimes\zeta_N^{a^{-1}b} \in \bQ(\zeta_N) \otimes \psi_N \]
by multiplication by $\zeta_N^a$.
We shall denote this $(a)$-isotypic eigenspace by $\psi_N(a) \subset \bQ(\zeta_N)\otimes\psi_N$.

Conversely, we recover the rational subspace $\bQ\zeta_N^b$ as being spanned by:
\[  \sum_{\gamma} \gamma(\zeta_N^b)\gamma(x_a), \]
where the sum is over $\gamma\in \Gal(\bQ(\zeta_N)/\bQ)$.
The vectors $\zeta_N^a$ for $a\in (\bZ/N\bZ)^\times$ form a rational basis for $\psi_N$.

Now consider the special case of the previous section where $F=\bQ(\zeta_N)$ and $E=\bQ(\zeta_N)^n$. Assume that $q = q_{E,\sigma,\lambda}$. Moreover, assume that the integral structure on $E=F^d$ is of the form $L=\oplus L_i$, where the $L_i$ are fractional ideals of $F^{(i)}$.
This requirement is equivalent to saying the integral structure is such that via $\mu_N \subset T_{F} \subset \Orth_q$ we find $\mu_N \subset \Orth_q(\bZ)$.

\begin{prop}
The cycle $\cD_{\phi}$ is the ramification divisor for $\mu_N$ under this action.
Moreover, locally near $D_\phi$ we have that:
\[ \cD_L = \cD_{\phi} \times \prod_{a\in (\bZ/N\bZ)^\times\setminus\{1\}} \bC^{r}(a-1), \]
where the action of $\mu_N$ on $\bC^{r}(a)$ is via $\chi^{a}$.
\end{prop}
\begin{proof}
We identify the tangent space near $\tau\in D_\phi$ with:
\[ \cT_{\cD_L,\tau} = \tau^{\perp}/\tau = \oplus_a (L\otimes\bC)(a)/\tau. \]
Without loss of generality (or rather by choice of $\zeta_N$) we may suppose $\tau$ is in the $\zeta_N$-eigenspace.
The above then becomes:
\[ \tau^{\perp}/\tau = \cT_{\cD_\phi}/\tau \underset{a\neq 1}\oplus( (L\otimes\bC)(a)/\tau). \]
We see that the action of $\mu_N$ on $(L\otimes\bC)(a)/\tau$ is by $\zeta_N^{a-1}$, where the $-1$ comes from the action on $\tau$.
We thus see that in a neighbourhood of $\tau$ around $\cD_{\phi}$ the group $\mu_N$ acts non-trivially, whereas it clearly acts trivially on $\cD_{\phi}$.
\end{proof}
\begin{rmk}
As with the previous case, points $\tau\in \cD_{\phi}$ may have other sources of ramification.
\end{rmk}

\subsection{Ramification at $\tau$}

We will now explain why the situations described above are in fact the only source of ramification.
Fix $\tau \in \cD_L$.
We define a lattice $S\subset L$ by setting: 
\[ S = (\{\Re(\tau),\Im(\tau)\}^\perp)^\perp.\]
Note that the lattice $S^\perp$ is a potentially $0$-dimensional negative-definite lattice.
We observe that $\tau \in \bC\otimes S$.
We wish to consider the stabilizer of $\tau \in \cD_L$.
This is precisely:
\[ \Gamma_\tau = \{ \gamma \in \Gamma \mid \text{there exists } \lambda_\gamma \in\bC^\times \text{ with }\gamma(\tau) = \lambda_\gamma\tau \}. \]
We immediately obtain a homomorphism $\chi_\tau : \Gamma_\tau \rightarrow \bC^\times$ given by $\chi_\tau(\gamma) = \lambda_\gamma$.

We have the following key results from \cite[Sec. 2.1]{GHSKodairaK3}.
\begin{prop}
\label{prop:RamISGHS}
With the above notation we see the following:
\begin{itemize}
\item There is an inclusion $\Gamma_\tau \subset \Gamma_S$.
\item The kernel $\ker(\chi_\tau)$ equals $\tilde{\Gamma}_{S^\perp}$.
\item The image of $\Gamma_\tau/\tilde{\Gamma}_{S^\perp}$ is a cyclic subgroup of $\overline{\Gamma}_S$.
\end{itemize}
\end{prop}
\begin{proof}
The first point follows immediately from the definition of $S$ and $\Gamma_S$.

To see the second point, notice that the inclusion $\tilde{\Gamma}_{S^\perp} \subset \ker(\chi_\tau)$ is apparent from the discussion of Section \ref{sec:SectionRNGHS}.
Now for the reverse inclusion, if $g\in \ker(\chi_\tau)$ and $x\in S$ we see:
\[ (\tau,x) = (g\tau,gx) = (\tau,gx). \]
This implies that:
\[ (\tau,x-gx) = (\overline{\tau},x-gx) = 0,\]
and thus, $x-gx \in S^\perp$. However, $S^\perp$ is negative-definite and thus:
\[  S\cap S^\perp = (S^\perp)^\perp \cap S^\perp = 0. \]

For the final point notice that: 
\[ \Gamma_\tau/\tilde{\Gamma}_{S^\perp} \simeq \chi_\tau(\Gamma_\tau) = \mu_{r_\tau} \subset \bC^\ast. \]
Thus the natural map:
\[ \Gamma_S \rightarrow \overline{\Gamma}_S \]
takes $\Gamma_\tau/\tilde{\Gamma}_{S^\perp}$ to a cyclic subgroup of $\overline{\Gamma}_S$.
\end{proof}
It follows from the proposition that the group $\Gamma_\tau/\tilde{\Gamma}_{S^\perp}$ gives an action of $\mu_{r_\tau}$ on $S$.

\begin{prop}
There are no trivial eigenvectors for the action of $\mu_{r_\tau}$ on $S$.
\end{prop}
\begin{proof}
Suppose $\vec{x}$ is a nontrivial eigenvector and that $g\in \mu_{r_\tau}$ is a nontrival element. Then we write:
\[ (\tau,x) = (g\tau,gx) = \chi_\tau(g)(\tau,x). \]
Likewise since $\overline{\tau}$ is also an eigenvector we find:
\[ (\overline{\tau},x) = (g\overline{\tau},gx) = \overline{\chi_\tau(g)}(\overline{\tau},x). \]
Therefore, $x\in S^\perp\cap S = \{0\}$.
\end{proof}

It follows from this proposition that $S = \phi_{r_\tau}^d$ as a representation of $\mu_{r_\tau}$.

\begin{prop}
\label{prop:PropNONTRIVIAL}
We can decompose $S = \phi_{r_\tau}^d$ in such a way that $q$ is non-degenerate on each factor and this is an orthogonal decomposition with respect to $q$.
\end{prop}
\begin{proof}
First we observe that we can proceed by induction provided there exists at least one non-degenerate factor.
Indeed, if $q|_{\phi_{r_\tau}}$ is non-degenerate it follows that $\mu_N$ stabilizes $(\phi_{r_\tau})^\perp$.
We may thus proceed inductively on $d$.

Next we observe that the restriction of $q$ is non-degenerate if and only if it is non-trivial. 
This follows from two key facts:
\begin{enumerate}
\item $\Gal(\bQ(\zeta_{r_\tau})/\bQ)$ acts transitively on eigenspaces, and
\item $b(x_a,x_b) = 0$ if $a\neq b^{-1}$.
\end{enumerate}
It follows that if $\varphi(r_\tau) > 2$, then $q|_{\phi_{r_\tau}}$ is non-degenerate since there are no isotropic spaces of size larger than $2$.

For the case of $\varphi(r_\tau) = 2$ it is not possible to have $d=1$. It follows that there exists a pair of $\phi_{r_\tau}$ such that the restriction of $q$ to $\phi_{r_\tau}^{(1)} \oplus \phi_{r_\tau}^{(2)}$ is nontrivial.
If $q$ restricts trivially to each factor, set $y_i^{(1)} = x_i^{(1)}+x_i^{(2)}$ and $y_i^{(2)} = x_i^{(1)}-x_i^{(2)}$.  The restriction of $q$ is then nontrivial on $\spann(y_i^{(j)}) \simeq \phi_{r_\tau}$.
This completes the argument.
\end{proof}

\begin{prop}
\label{prop:RamISSS}
If $\chi_\tau(\Gamma_\tau) \not\subset \{ \pm 1\}$ then $\tau$ is on a special cycle 
$\cD_{\phi}$ of $\cD_S$, where $F=\bQ(\chi(\Gamma_\tau))$. Hence, $\tau$ is on a generalized special cycle of $\cD_L$.
\end{prop}
\begin{proof}
Because the $\bQ$-span of $\phi_{r_\tau}(\mu_{r_\tau}) \subset \End( \phi_{r_\tau})$ is equal to $\bQ(\zeta_{r_\tau})$ we may extend the action of $\mu_{r_\tau}$ to one of $T_F$ on each factor.
This implies that we are in the setting of the previous section.
In particular, there exists a unique factor which is not negative-definite, and for it there exists a unique $\bR$-factor which is positive-definite.
\end{proof}

\begin{claim}
If $\chi_\tau(\Gamma_\tau) =\{ \pm 1\}$ then the image of $\Gamma_\tau$ acting on $\cD_{L,S}$ acts trivially on all of $\cD_{L,S}$.
\end{claim}
\begin{proof}
This follows since the entire space is the $(-1)$-eigenspace.
\end{proof}

\begin{rmk}
From Propositions \ref{prop:RamISGHS} and \ref{prop:RamISSS} it follows immediately that the ramification of $\cD_L$ consists entirely of the ramification along $\cD_{L,S}$ coming from $\tilde{\Gamma}_{S^\perp}$, and the ramification along $\cD_{\phi} \subset \cD_{L,S}$ coming from the action of $\mu_{N}$ on $\cD_{\phi}$.

Note though that if $\tilde{\Gamma}_{S^\perp} \neq \overline{\Gamma}_{S}$ then the quotient action by $\mu_N$ does not act trivially on the $S^\perp \otimes \bC$ component of the tangent space to $\cD_{L,S}$. 
This phenomenon can only arise if $L\neq S\oplus S^\perp$. 
\end{rmk}

\section{Concluding Remarks}

In this arcticle we have attempted to give an overview of the problems of constructing explicit toroidal compactifications for orthogonal locally symmetric spaces as well as that of computing dimension formulas for spaces of modular forms on them. Some areas that we have left unresolved in this are:
\begin{itemize}
\item An explicit projective regular cone decomposition for the orthogonal group.
\item Counting the number of inequivalent cusps.
\item Describing the geometry of the cusps.
\item In particular compute the Euler characteristic of the resulting compactifications of $k$-fold fiber products of the universal elliptic curve over the modular curve.

We expect that this should be immediately related to dimension formulas for modular forms on the underlying modular curve.
\item Study explicitly the effect of `elliptic' cycles on dimension formulas for non-neat level subgroups.
\item Explicit formulas.
\end{itemize}
Doing all of the above in complete generality is likely overly ambitious, one should thus probably restrict to special cases of interest.

\section*{Acknowledgements}
I would like to thank Victoria de Quehen and my Ph.D. Supervisor Prof. Eyal Goren for their help in editing various drafts of this paper. I would like to thank NSERC and the ISM for their financial support while doing this research.

\newpage

\providecommand{\MR}[1]{}
\providecommand{\bysame}{\leavevmode\hbox to3em{\hrulefill}\thinspace}
\providecommand{\MR}{\relax\ifhmode\unskip\space\fi MR }
\providecommand{\MRhref}[2]{  \href{http://www.ams.org/mathscinet-getitem?mr=#1}{#2}
}
\providecommand{\href}[2]{#2}

\end{document}